\documentclass[10pt,reqno]{amsart}

\usepackage{amsmath}
\usepackage{amsfonts}
\usepackage{amssymb}
\usepackage{amsthm}
\usepackage{mathtools}
\usepackage{mathrsfs}
\usepackage{latexsym}
\usepackage{verbatim} 
\numberwithin{equation}{section}
\usepackage{enumitem}

\usepackage{graphicx}
\usepackage{subfigure}
\usepackage{epstopdf}

\usepackage{empheq}

\usepackage{ifthen} 

\provideboolean{shownotes} 
\setboolean{shownotes}{false} 
\newcommand{\margnote}[1]{
\ifthenelse{\boolean{shownotes}}%
{\marginpar{\raggedright\tiny\texttt{#1}}}%
{}%
}
\newcommand{\hole}[1]{
\ifthenelse{\boolean{shownotes}}%
{\begin{center} \fbox{ \rule {.25cm}{0cm}
\rule[-.1cm]{0cm}{.4cm} \parbox{.85\textwidth}{\begin{center}
\texttt{#1}\end{center}} \rule {.25cm}{0cm}}\end{center}}
{}
}

\theoremstyle{plain}

\newtheorem{lemma}{Lemma}[section]
\newtheorem{theorem}[lemma]{Theorem}
\newtheorem{proposition}[lemma]{Proposition}
\newtheorem{corollary}[lemma]{Corollary}

\theoremstyle{definition}

\newtheorem{remark}[lemma]{Remark}
\newtheorem{definition}[lemma]{Definition}

\theoremstyle{remark}

\usepackage{cite} 

\usepackage[colorlinks=true,urlcolor=blue,
citecolor=red,linkcolor=blue,linktocpage,pdfpagelabels,
bookmarksnumbered,bookmarksopen]{hyperref}

\usepackage{cleveref}

\newcommand{\Id}{\mathrm{I}}

\newcommand{\bm}{\mathbf{m}}
\newcommand{\bH}{\mathbf{H}}
\newcommand{\bh}{\mathbf{h}}

\DeclareFontFamily{U}{mathx}{}
\DeclareFontShape{U}{mathx}{m}{n}{<-> mathx10}{}
\DeclareSymbolFont{mathx}{U}{mathx}{m}{n}
\DeclareMathAccent{\widehat}{0}{mathx}{"70}
\DeclareMathAccent{\widecheck}{0}{mathx}{"71}

\newcommand{\E}{\mathbb{E}}
\newcommand{\R}{\mathbb{R}}
\newcommand{\C}{\mathbb{C}}
\newcommand{\Z}{\mathbb{Z}}

\newcommand{\cN}{{\mathcal{N}}}
\newcommand{\cT}{{\mathcal{T}}}
\newcommand{\cE}{{\mathcal{E}}}
\newcommand{\cL}{{\mathcal{L}}}
\newcommand{\ocL}{{\overline{\cL}}}
\newcommand{\cD}{{\mathcal{D}}}

\newcommand{\cQ}{{\mathcal{Q}}}

\newcommand{\cS}{{\mathcal{S}}}

\newcommand{\cP}{{\mathcal{P}}}

\newcommand{\cA}{{\mathcal{A}}}

\newcommand{\cB}{{\mathcal{B}}}

\newcommand{\conv}{\mathrm{conv}}

\newcommand{\intconv}{\mathrm{int}\,\mathrm{conv}}

\newcommand{\tbet}{\widetilde{\beta}}

\renewcommand{\Re}{\mathrm{Re}\,} 
\renewcommand{\Im}{\mathrm{Im}\,}

\usepackage{scalerel}

\newcommand{\ep}{\epsilon}

\newcommand{\brt}{\bar{\theta}}

\newcommand{\brtp}{{\brt \,}'}

\newcommand{\varep}{\varepsilon}

\newcommand{\re}[1]{\textrm{Re}\left(#1\right)}

\newcommand{\pld}[2]{\left \langle #1 \, , #2 \right \rangle_{L^2}}
\newcommand{\phu}[2]{\left \langle #1 \, , #2 \right \rangle_{H^1}}

\newcommand{\nld}[1]{\left \| #1 \right \|_{L^2}}
\newcommand{\nhu}[1]{\left \| #1 \right \|_{H^1}}
\newcommand{\nhd}[1]{\left \| #1 \right \|_{H^2}}
\newcommand{\nwse}[1]{\left \| #1 \right \|_{H^1\times L^2}}

\newcommand{\res}[1]{\rho\left ( #1 \right )}

\DeclareMathOperator{\opl}{\mathcal{L}}

\DeclareMathOperator{\ran}{Ran}

\begin{document}
\title[Stability of moving N\'{e}el walls in ferromagnetic thin films]{Stability of moving N\'{e}el walls in ferromagnetic thin films}

\author[A. Capella]{Antonio Capella}

\address{{\rm (A. Capella)} Instituto de Matem\'aticas\\Universidad Nacional Aut\'onoma de M\'exico\\Circuito Exterior s/n, Ciudad Universitaria\\C.P. 04510 Cd. de M\'{e}xico (Mexico)}

\email{capella@matem.unam.mx}

\author[C. Melcher]{Christof Melcher}

\address{{\rm (C. Melcher)} 
Lehrstuhl f\"ur Angewandte Analysis \\ RWTH Aachen\\D-52056 Aachen (Germany)}

\email{melcher@rwth-aachen.de}

\author[L. Morales]{Lauro Morales}

\address{{\rm (L. Morales)} Instituto de Investigaciones en Matem\'aticas Aplicadas y en Sistemas\\Universidad Nacional Aut\'onoma de M\'exico\\Circuito Escolar s/n, Ciudad Universitaria\\C.P. 04510 Cd. de M\'{e}xico (Mexico)}

\email{lauro.morales@iimas.unam.mx}

\author[R. G. Plaza]{Ram\'on G. Plaza}

\address{{\rm (R. G. Plaza)} Instituto de Investigaciones en Matem\'aticas Aplicadas y en Sistemas\\Universidad Nacional Aut\'onoma de M\'exico\\Circuito Escolar s/n, Ciudad Universitaria\\C.P. 04510 Cd. de M\'{e}xico (Mexico)}

\email{plaza@aries.iimas.unam.mx}

\begin{abstract}
This paper studies moving 180-degree N\'eel walls in ferromagnetic thin films under the reduced model for the in-plane magnetization proposed by Capella, Melcher and Otto \cite{CMO07}, in the case when a sufficiently weak external magnetic field is applied. It is shown that the linearization around the moving N\'eel wall's phase determines a spectral problem that is a relatively bounded perturbation of the linearization around the static N\'eel wall, which is the solution when the external magnetic field is set to zero and which is spectrally stable. Uniform resolvent-type estimates for the linearized operator around the static wall are established in order to prove the spectral stability of the moving wall upon application of perturbation theory for linear operators. The spectral analysis is the basis to prove, in turn, both the decaying properties of the generated semigroup and the nonlinear stability of the moving N\'eel wall under small perturbations, in the case of a sufficiently weak external magnetic field. The stability of the static N\'eel wall, which was established in a companion paper \cite{CMMP23}, plays a key role to obtain the main result.
\end{abstract}

\keywords{Traveling N\'eel walls, ferromagnetic thin films, spectral stability, resolvent estimates.}

\subjclass[2020]{35B35, 35Q60, 35C07, 82D40, 78M22, 47A10.}

\maketitle
\setcounter{tocdepth}{1}



\section{Introduction}
\label{secintro}

The emergence of domain walls and the study of their dynamics constitute one of the most fundamental topics in the theory of ferromagnetic materials. The term \emph{domain wall} refers to a narrow transition region between opposite magnetization vectors inside a ferromagnet (cf. Hubert and Sch\"afer \cite{HuSch98}). In order to describe the evolution of the magnetization inside a ferromagnetic material, Landau and Lifshitz \cite{LanLif35} introduced in 1935 a model system of equations (later reformulated and re-derived by Gilbert \cite{Gilb55}) known as the Laundau-Lifshitz-Gilbert (LLG) micromagnetic modeling framework. The model relates the observed magnetization patterns to the result of minimizing a micromagnetic energy functional, and it is posed in terms of a damped gyromagnetic precession of a (unit) magnetization vector field. The LLG model and its extensions support the description of the magnetization in response to external force fields and damping terms as well (cf. \cite{Gilb04,BrKOh12}). It was N\'eel \cite{Neel55} who first pointed out that, when the thickness of a certain ferromagnet becomes sufficiently small with respect to the exchange characteristic length, it then becomes energetically favorable for the magnetization to rotate in the thin film plane, giving rise to a \emph{N\'eel wall}. This wall separates two opposite magnetization regions by an in-plane rotation oriented along an axis. 

In order to study the dynamics of N\'eel walls, Capella \emph{et al.} \cite{CMO07} proposed a one-dimensional thin film reduction of the micromagnetic energy (previously outlined by Muratov and Osipov \cite{MuOs06} for numerical purposes), which profits from the gyrotropic nature of the LLG model and from the shape of anisotropy effects from stray-field interactions, resulting into a thin-film layer equation for the in-plane magnetization's phase. The effective equations encompass a wave-type dynamics for such N\'eel wall's phase. The authors in \cite{CMO07} proved the existence of wave profile solutions to the resulting model in both the static (in the absence of an external magnetic field) and in the dynamical (when a weak in-plane external magnetic field is applied) cases. These solutions effectively describe the dynamics of N\'eel walls in the thin film limit (see also Chermisi and Muratov \cite{ChMu13} and Muratov and Yan \cite{MuYa16} for further information).

From a mathematical perspective, micromagnetics pose many challenging problems, ranging from the calculus of variations to stochastic analysis, due to the non linear, non-local nature of the model equations. Regarding their dynamics, one of the main problems is to understand the behavior of these domain walls under small perturbations. Therefore, the dynamical stability of such structures is a fundamental feature, not only to validate the mathematical model, but also to enhance the numerical simulations performed by physicists and engineers in order to design new ferromagnetic materials (see, e.g., \cite{LabBer99}). As far as we know, there are few results on the stability of ferromagnetic domain walls under a dynamical point of view reported in the scientific literature; for an abridged list of works the reader is refereed to Carbou and Labb\'e \cite{CarLab06}, Carbou \cite{Carb10}, Carbou, Massaoui and Rachi \cite{CaMR22}, Krukowski \cite{Kruk87}, Takasao \cite{Tak11} and some of the references cited therein. With the exception of the recent result by Carbou \emph{et al.} \cite{CaMR22} (which studies the spectral stability of a constant coefficient linearized operator around steady states in ferromagnetic rings), most of the aforementioned works establish \emph{a priori} energy estimates on the evolution equations governing eventual perturbations of the profile solutions and study their decay properties. It is to be observed that energy methods strongly rely on the intrinsic structure of the model equations. 

In a companion paper \cite{CMMP23}, we established the nonlinear stability of the static 180-degree N\'eel wall profile under the reduced wave-type dynamics for the in-plane magnetization proposed by Capella \emph{et al.} \cite{CMO07}, which is the energy minimizer \emph{in the absence of an external magnetic field}. For that purpose, in \cite{CMMP23} we adopted a different strategy to tackle the stability problem. Motivated by the (now standard) ``spectral stability implies nonlinear stability'' methodology (see, for example, the seminal works by Alexander \emph{et al.} \cite{AGJ90}, Pego and Weinstein \cite{PW94,PW92b}, Zumbrun and Howard \cite{ZH98}, and the surveys by Sandstede \cite{San02} and Kapitula and Promislow \cite{KaPro13}), we performed the first rigorous proof of the nonlinear stability of the static N\'eel wall based on a spectral study of the linearization around the wall phase's profile. The latter is a non-local, block operator matrix posed on a suitable energy space. In our analysis, we proved that this operator is spectrally stable, that is, its spectrum is contained in the stable half-plane of complex numbers with strictly negative real part, except for a simple zero eigenvalue associated to translations of the profile. Two important features of the linearization about the static wall are, first, the presence of a \emph{spectral gap} (that is, a positive distance from the zero eigenvalue to the rest of the spectrum), which allows to conclude the exponential decay of the semigroup; and, second, the simplicity of the eigenvalue zero, which allows to nonlinearly modulate perturbations depending on translations alone. The main difficulty of the analysis, however, consisted on the non-local nature of the operator, for which the location of the essential spectrum involved the establishment of a relative-compactness property based on $L^2$-equicontinuity of Fourier operator symbols (see \cite{CMMP23} for details). Outside the one-dimensional eigenspace generated by the derivative of the static N\'eel profile, the infinitesimal semigroup generated by this linear block operator decays exponentially. Then, upon application of Gearhart-Pr\"uss theorem (cf. \cite{CrL03,EN00}) and of an abstract nonlinear stability result by Lattanzio \emph{et al.} \cite{LMPS16} (see also \cite{Sat76}), we were able to deduce the nonlinear stability of the static N\'eel profile from the previous spectral information.

In this paper we pose the following question: what happens to the N\'eel wall profile in the presence of an external magnetic field? Capella \emph{et al.} \cite{CMO07} have already proved that, if the applied external magnetic field, $\mathbf{H} = H \mathbf{e}_y$, is sufficiently weak with $|H| \ll 1$, where $H$ denotes its intensity, then there exists a traveling wave profile solution associated to a \emph{moving N\'eel wall}. This solution is a phase profile that travels with speed $c$, which is also small, $c = O(|H|)$. In this work we study the dynamical stability of such a moving N\'eel wall profile under small perturbations of the wall's phase in appropriate energy spaces, once the applied magnetic field has been fixed. We employ the same methodology: we linearize around the  moving profile and analyze the spectrum of the resulting operator. At this point it is important to remark, however, that the analysis presented here \emph{is not incremental} (that is, we do not perform nor repeat the same steps as in the case of the static operator) but rather complementary: we depart from the stability result for the static operator and prove that the linearization around a moving N\'eel wall with speed $c$, namely, a linearized non-local block operator $\cA_c$, is a relatively bounded perturbation of the linearization around the static N\'eel wall, $\cA$. This observation allows us to apply standard perturbation theory of linear operators (cf. Kato \cite{Kat80}) to conclude the spectral stability of the linearized operator for the moving wall. For that purpose, we prove that the resolvent of the operator $\cA_c$ can be approximated by that of the static operator $\cA$. 

Despite the simplicity of this approach, the analysis is far from trivial. Indeed, our conclusions depend on the establishment of certain resolvent-type estimates for the static operator $\cA$. The whole central Section \ref{secresests} is devoted to accomplish this task. Although the analysis is elementary and the estimations are direct, the latter are, however, quite convoluted. For instance, it is necessary to keep track of resolvent bounds on different regions of the resolvent set and on a small neighborhood of the origin (see the proof of Theorem \ref{lem:spectral_gap} below). Upon application of standard perturbation theory for linear operators we then conclude the spectral stability of the moving wall for small values of $c$ (or, equivalently, of the applied magnetic field). The fact that the magnetic field determines uniquely the speed of the wall implies that the manifold generated by the traveling profile is  one dimensional and we only need to modulate perturbations via translations of the profile. From a spectral viewpoint, this is tantamount to the translation zero eigenvalue of the linearized operator being simple. Perturbation theory also yields that the spectrum of the perturbed operator remains close to that of the static operator. In this fashion, we recover spectral stability with a spectral gap and the simplicity of the translation eigenvalue in the case of small applied magnetic fields. In other words, we prove the persistence of the spectral stability properties of the static wall's phase in the case of a moving wall under a weak external field. We then apply Lumer-Phillips theorem to generate an exponentially decaying semigroup outside the one-dimensional eigenspace. We use this property and the simplicity of the eigenvalue to prove, in turn, the nonlinear stability of the moving wall under small perturbations, just like in the static case \cite{CMMP23}. This is perhaps the most standard part of the analysis. We include it for the sake of completeness; however, we gloss over many details which can be found somewhere else.

The contributions of this paper can be summarized as follows:
\begin{itemize}
\item[--] we establish new regularity properties for the moving N\'eel wall which are required for the stability analysis;
\item[--] we pose the spectral stability problem for the moving wall and, motivated by energy minimization arguments (see \cite{CMMP23}), we select the space of $H^1$ perturbations for the phase;
\item[--] we prove that the spectral equation can be recast as a relatively bounded perturbation of the linearized operator of the static N\'eel wall of the form $\cA_c = \cA + \cB_c$, where $\cA_c $ is the linearized operator around the moving wall with speed $c = O(|H|)$, $|H| \ll 1$, $\cB_c$ is an $\cA$-bounded operator, and $\cA$ denotes the linearized operator around the static N\'eel wall but now evaluated in a (Galilean) moving frame of the form $z = x-ct$ (see Lemma \ref{lem:growth_bounds} below);
\item[--] we establish uniform resolvent-type estimates on the static oerator $\cA$ which allows to conclude the spectral stability of the linearized operator $\cA_c$ upon application of perturbation theory for linear operators (see Theorems \ref{lem:spectral_gap} and \ref{theospectralstab} below);
\item[--] we prove that the operator $\cA_c$ is the infinitesimal generator of a $C_0$-semigroup which is exponentially decaying outside the one-dimensional eigenspace associated to the eigenvalue zero; and,
\item[--] we close the analysis by showing that the moving N\'eel wall is nonlinearly stable under small perturbations of its phase in the energy space.
\end{itemize}

\subsection*{Plan of the paper}

The remainder of the paper is structured as follows. In order to be able to state our main result (see Theorem \ref{maintheorem} below), in Section \ref{secassump} we recall the thin film reduction of the LLG model and describe the stability results for the static wall. Section \ref{secprel} contains some important properties of the moving N\'eel wall's phase and its relation to the static profile, which are needed for the stability analysis. Section \ref{secperturb} establishes the perturbation equations to be studied and sets up the associated spectral problem. It also contains the relative boundedness result that motivates our methodological approach (Lemma \ref{lem:growth_bounds}). The central Section \ref{secresests} is devoted to establish resolvent-type estimates for the static operator which allow to apply perturbation theory and to conclude the spectral stability of the moving wall (a short Appendix \ref{secappendix} contains some pointwise estimates that are needed for the proof). Section \ref{secsemiggen} focuses on the (standard) generation of the $C_0$-semigroup and on its decaying properties. The final Section \ref{sec:nonlinear_stability} establishes the nonlinear stability of the moving N\'eel wall. We close the exposition with some final remarks (Section \ref{secdiscuss}).

\subsection*{Notations}
We denote the spaces  $L^2(\R, \C), \ H^1(\R, \C)$ and $H^2(\R, \C)$ of complex-valued functions as $L^2, \ H^1$ and $H^2$. Their real-valued counterparts are denoted as $L^2(\R), \ H^1(\R)$ and $H^2(\R)$, respectively. 
We write $\conv(\Gamma)$ to denote the convex hull of a given set $\Gamma\subset \C$ and $\intconv (\Gamma)$ to denote its interior. The set of unitary vectors in $\R^n$ is denoted by $\mathbb{S}^{n-1}$. For any number or complex-valued function, the operation $(\cdot)^*$  denotes  complex conjugation. The operators $\widehat{\cdot}:L^2\to L^2$ and $\widecheck{\cdot}:L^2\to L^2$ stand for the Fourier transform and its inverse, respectively. Also, $\xi$ represents the variable in the Fourier domain. The half-Laplacian is defined by the relation $(-\Delta)^{1/2}u = (|\xi|\widehat{u})\,\widecheck{}$, and $\|u\|_{\dot{H}^{1/2}}$ denotes the fractional $H^{1/2}$-seminorm of the function $u\in L^2$ given by $\|u\|_{\dot{H}^{1/2}}:=\nld{|\xi|^{1/2}\widehat{u}}$. For any linear, closed and densely defined operator operator, $\cL : D(\cL) \subset X \to Y$, with $X$, $Y$ Banach spaces and domain $D(\cL)\subset X$, the resolvent set, $\rho(\cL)$, is defined as the set of complex numbers $\lambda \in \C$ such that $\cL - \lambda$ is injective and onto, and $(\cL - \lambda)^{-1}$ is a bounded operator. The spectrum of $\cL$ is the complex complement of the resolvent, $\sigma(\cL) = \C \backslash \rho(\cL)$.

\section{Equations, assumptions and main result}
\label{secassump}

In this Section we describe the micromagnetic equations in the thin film limit, recall the previous stability results of the static N\'eel wall and state the main result of this paper.

\subsection{The micromagnetic model and N\'{e}el walls}
\label{secLLG}

The time evolution of the magnetization distribution on a ferromagnetic body, $\widetilde{\Omega} \subset \R^3$, is governed by the Landau-Lifshitz-Gilbert (LLG) equation (cf. \cite{LanLif35,Gilb55}):
\begin{equation}
\label{eqLLG}
\bm_t+ \bm \times (\gamma \bH_{\mathrm{eff}} - \alpha \bm_t) = 0, 
\end{equation}
where $\bm : \widetilde{\Omega} \times (0,\infty) \to \mathbb{S}^2 \subset \R^3$ is the magnetization field, $\alpha > 0$ is a non-dimensional damping coefficient (Gilbert factor) and $\gamma > 0$ is the (constant) absolute value of the gyromagnetic ratio with dimensions of frequency (see, e.g., Gilbert \cite{Gilb04}). The effective field, $\bH_{\mathrm{eff}} = \bh - \nabla \E(\bm)$, consists of the applied field $\bh$ and the negative functional gradient of the micromagnetic interaction energy $\E(\bm)$, which, in the absence of external fields, is given by
\[
 \E(\bm) = \frac{1}{2}\Big( d^2 \int_{\widetilde{\Omega}} |\nabla \bm|^2 \, dx + \int_{\R^3} |\nabla U|^2 + Q \int_{\widetilde{\Omega}} \Phi(\bm) \, dx \Big).
\]
The parameter $d> 0$ is the exchange length and the \textit{stray field}, $\nabla U$, is defined uniquely via the distribution equation $\Delta U = \textrm{div}\,(\bm \chi_{\widetilde{\Omega}})$ ($\chi_A$ denotes the indicator function of the set $A$). The last integral models crystalline anisotropies via a penalty energy, for which $\Phi$ acts as a penalty function and it has usually the form of an even polynomial in $\bm \in \mathbb{S}^2$. The parameter $Q>0$ measures the relative strength of anisotropy penalization against stray-field interaction.

In the thin-film regime, that is, when $\widetilde{\Omega} = \Omega \times (0,\delta)$ with $\Omega \subset \R^2$ and $0 < \delta \ll d$, it is usually assumed that the magnetization is independent of the $x_3$ variable and $\ell$-periodic in the ${\bf e}_2$ direction, namely,
\[
{\bf m}(x_1,x_2+\ell) = \bm(x_1,x_2) \quad\text{for any } x = (x_1,x_2)\in\R^2.
\]
Moreover, it can be justified that the material underlies uniaxial anisotropy in the ${\bf e}_2$ direction, with $\Phi({\bf m}) = 1- m_2^2$. Under the appropriate scalings, Capella \emph{et al.} \cite{CMO07} proved that the magnetization, ${\bf m} = (m,0)$ with $m = (m_1,m_2)$, is a solution to the following variational problem
 \begin{equation}
 \label{limit-varppio}
 \begin{aligned}
E_0 (m) &= \tfrac{1}{2} \left({\mathcal Q} 
\Vert m'\Vert_{L^2(\R)}^2 + \Vert m_1 \Vert_{\dot{H}^{1/2}(\R)}^2 + \Vert m_1\Vert_{L^2(\R)}^2    \right) \to \min,\\ 
 &m : \R \to {\mathbb S}^2, \qquad \text{with} \;\;  
 m(\pm\infty) =(0,\pm 1),
 \end{aligned}
\end{equation}
where $' = d/d x_1$ and the constant $\mathcal{Q} > 0$ is a rescaled relative strength $Q$. Since the left translation operator is an isometry in $L^2$, the expression of $E_0(m)$ is invariant under spatial translations. This invariance is inherited by the energy, yielding that minimizers of \eqref{limit-varppio} are unique up to translations. Despite this invariance, $E_0(m)$ is a strictly convex functional on $m_1$  because  $|m'|^2=(m_1')^2/(1-m_1^2)$. Thus, the variational principle \eqref{limit-varppio} has a minimizer for any ${\mathcal Q}>0$. The minimizer that satisfies $m_1(0)=1$ is called the {\em N\'eel wall profile}. We refer to $E_0(m)$ as the {\em N\'eel wall energy}. In other words, in the thin film regime the normal component of the magnetization for any ferromagnetic sample is penalized by the geometry in such a way that this component must vanish as the height of the sample does (see \cite{CMO07,GC04,Melc10} for further details).

Consequently, the in-plane magnetization is completely determined by its phase $\theta \in (-\pi/2,\pi/2)$ through the relation $m = (m_1, m_2) = (\cos \theta, \sin \theta)$, and the variational problem that defines a N\'eel wall becomes the following variational problem for the N\'eel wall's phase,
\begin{equation}
 \label{varprob}
 \begin{aligned}
 \mathcal{E}(\theta) &= \frac{1}{2} \big( \|\theta'\|_{L^2}^2 + \|\cos \theta\|^2_{{\dot H}^{1/2}} + \|\cos \theta\|_{L^2}^2 \big) \; \rightarrow \;  \min\\ 
 \theta : \R &\to (-\pi/2,\pi/2), \qquad \text{with} \;\; \theta(\pm\infty) = \pm \pi/2
 \end{aligned}
\end{equation}
(for details, see Capella \emph{et al.} \cite{CMO07}). For simplicity and without loss of generality, here we have assumed that $\cQ \equiv 1$. We keep such normalization for the rest of the paper.

The magnetization and its phase are, of course, heavily influenced by the presence of an applied external magnetic field. Indeed, in the case where the latter points towards one of the end-states determined by the anisotropy, that is, $\mathbf{h} = H \mathbf{e}_2$, then it can be proved (cf.\cite{CMO07}) that the dynamical equation for the phase $\theta$ is given by
\begin{equation}
 \label{reddyneq}
 \left\{ \ \ 
\begin{aligned}
&\partial_t^2 \theta + \nu \partial_t \theta + \nabla \cE(\theta) = H\cos\theta, \\
&\theta(-\infty,t) =-\pi/2,\quad \theta(\infty,t) =\pi/2,\\
&\theta(x,0) =\theta_0(x) ,\quad \partial_t\theta(x,0) =v_0(x),
\end{aligned}
\right.
\end{equation}
where $\theta_0$ and $v_0$ are some initial conditions, $H \in \R$ is the scalar parameter that measures the external magnetic field strength and  $\cE(\theta)$ is the effective energy appearing in \eqref{varprob}.

Both equations \eqref{varprob} and \eqref{reddyneq} constitute an effective model that describes the dynamics of the magnetization in a ferromagnetic thin film. They are strictly derived from electromagnetic theory and the LLG equation. We remark that, in contrast with the LLG model, equation \eqref{reddyneq} has a second time derivative which is a consequence of the magnetization vector being unitary, which is, in turn, an holonomic constrain. In addition, equation \eqref{reddyneq} not only has a damping term as LLG-equation does, but it also contains a second order spatial derivative emerging form the term $\nabla \cE(\theta)$. In this fashion, Capella \emph{et al.} \cite{CMO07} showed that the in-plane magnetization of a ferromagnetic thin film exhibits some wave-type dynamics and proved the existence of traveling wave solutions (N\'eel wall profiles) under small values of the applied field, $ 0 \leq |H| \ll 1$, including $H = 0$, as we describe below.

\subsection{The static N\'{e}el wall profile in soft magnetic thin films}

In the absence of external magnetic field with $H \equiv 0$, the Cauchy problem \eqref{reddyneq} has a unique odd, monotone increasing and smooth steady state, $\brt = \brt(x)$, such that ${\brt \,}' = {\brt}'(x) \in H^k$ for any $k \in \Z$, $k \geq 0$, known as the \emph{static N\'eel wall's phase profile}. Also, equation \eqref{reddyneq} has a one-parameter group symmetry due to translation invariance. This fact and the positiveness of the damping coefficient $\nu > 0$ imply the orbital stability of the steady state $\brt$, as it has been recently proved in \cite{CMMP23}. Therein, we studied the perturbation equation of \eqref{reddyneq} around $\brt$, by recasting \eqref{reddyneq} with $H = 0$ as an evolution system, namely,
\begin{equation}
\label{NLsyst}
\partial_t \begin{pmatrix}
            u \\ v
           \end{pmatrix} = \begin{pmatrix}0 & \Id \\ - \cL & - \nu \Id\end{pmatrix} \begin{pmatrix}u \\ v\end{pmatrix} + \begin{pmatrix}0 \\ \nabla \cE(\brt + u) - \cL u\end{pmatrix},
\end{equation}
where $\cL:H^1\to L^2$ is the $H^1$-restriction of the $L^2$-gradient of $\nabla \cE$ around the steady state $\brt$. Here the variable $u(\cdot,t) \in H^1$ for all $t > 0$ denotes a small perturbation of the N\'eel wall's phase profile and $v = \partial_t u$. 
A straightforward computation shows that $\nabla \cE(\theta)$ is given in terms of the nonlocal linear operator $(1+(-\Delta)^{1/2}):H^{1}\to L^2$. 
\begin{proposition}\label{lem:cT}
For any $k = 0,1, 2$, the nonlocal linear operator $\cT:=(1+(-\Delta)^{1/2}):H^{k+1}\to H^k$ is bounded.
\end{proposition}
\begin{proof}
Follows immediately from its definition in terms of the Fourier transform. Indeed, for each $k = 0,1,2$, there holds
\[
\begin{aligned}
\| \cT u \|_{H^k}^2 = \| (1 +(-\Delta)^{1/2}) u \|_{H^k}^2 &= \int_\R (1+|\xi|^2)^k \big| \big( (1 +(-\Delta)^{1/2}) u\big)\,\widehat{}\,(\xi)\big|^2 \, d\xi \\
&\leq 2 \int_\R (1+|\xi|^2)^{k+1} | \widehat{u}(\xi) |^2 \, d\xi = 2 \| u \|_{H^{k+1}}^2,
\end{aligned}
\]
yielding 
\[
\| \cT \| = \sup_{\| u \|_{H^{k+1}} =1} \| \cT u \|_{H^k} < \infty.
\]
\end{proof}
\begin{remark}
\label{remnotation}
In order to simplify the notation, for any real function $\phi = \phi(x)$ depending on the spatial coordinate $x \in \R$ we write $s_\phi = s_\phi(x)$ to denote the function $x\mapsto\sin \phi(x)$, and $c_\phi = c_\phi(x)$ to denote the function $x\mapsto \cos\phi(x)\cT(\cos\phi(x))$. We shall keep this notation for the rest of the paper.
\end{remark}

Therefore, it can be shown (cf. \cite{CMO07,CMMP23}) that the $L^2$-gradient of $\nabla \cE$ around the steady state $\brt$, that is, the linearized operator $\cL$ around the static N\'eel profile acting on the energy space $L^2$, is given by
\begin{equation}
        \label{eq:previousL}
        \left\{
	   \arraycolsep=1pt\def\arraystretch{1.5}
        \begin{array}{l}
        \cL : L^2 \to L^2,\\
        D(\cL) = H^2,\\
        \cL u := - \partial^2_x u +  s_{\brt} \cT (s_{\brt} u) - c_{\brt} u, \qquad u \in D(\cL).
        \end{array}
        \right.
    \end{equation}
    
\begin{remark}
\label{remLinH1}    
It is to be observed that the operator $\cL$ that appears in \eqref{NLsyst} refers to the restriction of $\cL$ to $H^1$, namely $\cL_{|H^1}$, which we write again as $\cL$ with a slight abuse of notation. Nonetheless, its spectral properties and its closedness remain (see Remark 5.2 in \cite{CMMP23}).
\end{remark}

One of the key points proved by Capella \emph{et al.} \cite{CMMP23} is that the solutions to the linear version of \eqref{NLsyst} constitute a $C_0$-semigroup. The following result summarizes the main spectral properties of the linearization around the static N\'eel wall's phase.
\begin{theorem}[Capella \emph{et al.} \cite{CMMP23}]
\label{thm:steady_result}
Let $\nu>0$ be fixed and let $\cA:H^1\times L^2\to H^1\times L^2$, with dense domain $D(\cA) = H^2\times H^1$, be the block operator given by
\begin{equation}\label{eq:A}
\begin{pmatrix}
            u \\ v
           \end{pmatrix} \mapsto \begin{pmatrix}0 & \Id \\ - \cL & - \nu \Id\end{pmatrix} \begin{pmatrix}u \\ v\end{pmatrix}. 
\end{equation}
Then the following properties hold:
\begin{itemize}
	\item[\rm{(a)}] $\lambda = 0$ is a simple and isolated eigenvalue of $\cA$ with eigenvector $\Theta_0 := (\brtp,0) \in D(\cA)$.
   	 \item[\rm{(b)}] There exists $\zeta(\nu)> \tfrac{1}{2}\nu > 0$ such that 
	 \[
	 \sigma(\cA) \subset \{ 0 \} \cup \big\{ \lambda \in \C \, : \, \Re \lambda \leq - \zeta(\nu) < 0 \big\}.
	 \]
    	\item[\rm{(c)}] $\cA$ is the infinitesimal generator of a $C_0$-semigroup $\{e^{t\cA}\}_{t\geq 0}$ of quasicontractions.
    	\item[\rm{(d)}] Let $\Phi_0 := (\nu\brtp,\brtp) \in D(\cA)$. Then the projection operator $\cP$ defined on $H^1 \times L^2$ and given by 
\begin{equation}
\label{defofP}
U\in H^1\times L^2 \mapsto \cP U := U - \frac{\langle U, \Phi_0 \rangle_{L^2 \times L^2}}{\langle \Theta_0, \Phi_0 \rangle_{L^2 \times L^2}}\Theta_0,
\end{equation}
commutes with $\cA$ and satisfies 
\[
\ran(\cA)\subset \ran(\cP)=\{F\in H^1\times L^2 \ |\ \langle F, \Phi_0\rangle_{L^2 \times L^2}=0\}.
\]
\end{itemize}
Moreover, if $(H^1\times L^2)_{\perp}$ denotes the range of $\cP$, 
\begin{equation}
\label{ranofP}
(H^1\times L^2)_{\perp}:= \ran(\cP), 
\end{equation}
then the restriction of the operator $\cA$ on $(H^1\times L^2)_{\perp}$, denoted as $\cA_\perp := \cA_{|(H^1\times L^2)_{\perp}}$ and with dense domain $D(\cA_\perp) = (H^2\times H^1)\cap (H^1\times L^2)_{\perp}$, is the infinitesimal generator of an exponentially decaying $C_0$-semigroup $\{e^{t\cA_\perp}\}_{t\geq 0}$ and its spectrum is stable with a spectral gap,
\begin{equation}
\label{staticstable}
\sigma(\cA_\perp) \subset \big\{ \lambda \in \C \, : \, \Re \lambda \leq - \zeta(\nu) < 0 \big\}.
\end{equation}
\end{theorem}

\begin{remark}
A few observations about the projector $\cP$ are in order. The reader might have already noticed that the inner product appearing in formula \eqref{defofP} is not the standard inner product of the base space $H^1 \times L^2$. The operator $\cP$ is, however, the standard projection of $H^1 \times L^2$ onto the eigenspace $\text{span} \, \{ \Theta_0 \}$. This holds because, as stated in \cite{CMMP23}, the eigenfunction associated to the formal adjoint $\cA^*$ is not $\Phi_0 = (\nu\brtp,\brtp)$ but $\check{\Phi}_0 = (\nu (1 + (-\Delta))^{-1} \brtp, \brtp) \in H^1 \times L^2$ and it can be proved that
\[
\langle U, \check{\Phi}_0 \rangle_{H^1 \times L^2} = \langle U, \Phi_0 \rangle_{L^2 \times L^2},
\]
for all $U \in H^1 \times L^2$ (see Corollary 6.9 and Lemma 6.10 in \cite{CMMP23}). Hence we obtain formula \eqref{defofP}, which provides a simpler characterization of $\ran(\cP)$ and which was crucial to deduce the spectral stability of the static N\'eel wall, via a direct spectral analysis of the operator $\cA$. Compare \eqref{defofP} with the standard expression of the corresponding projector onto the eigenspace generated by the \emph{moving} N\'eel wall, formula \eqref{eq:projector} below. In the moving case we do not pursue a simpler expression because the spectral stability of the linearized operator about a moving wall is analyzed indirectly (upon application of standard perturbation theory of linear operators).
\end{remark}

In the companion paper \cite{CMMP23}, we were able to prove the nonlinear (orbital) stability of the static N\'eel wall under small perturbations. We used the spectral information of Theorem \ref{thm:steady_result} in a key way. We shall profit from this result to examine the case of \emph{moving} N\'eel walls, which are related structures that appear when the external magnetic field is switched on.

\subsection{The moving N\'eel wall and statement of the main result}

Now let us consider the case when an external magnetic field is applied. By assuming the existence of traveling wave solutions to equation \eqref{reddyneq} of the form $\theta(x,t) = \psi(x-ct)$ for $c\in \R$,  a direct computation shows that the the traveling profile $\psi$ must satisfy 
\begin{equation}\label{travelling_profile}
\begin{aligned}
&c^2\psi'' - \nu c\psi' +\nabla \cE(\psi) = H\cos \psi,\\
&\psi(-\infty) =-\pi/2,\quad \psi(\infty) =\pi/2,\\
\end{aligned}
\end{equation}
where $' = d/dz$ denotes differentiation with respect to the moving space variable of translation, $z = x - ct$. We refer to the constant $c\in \R$ as the {\em propagation} speed. Upon an application of the implicit function theorem, the authors of \cite{CMO07} proved the existence of a moving N\'eel wall's phase for sufficiently small values of the applied magnetic field.
\begin{theorem}[Capella \emph{et al.} \cite{CMO07}]
\label{thm:CMO07}
There exists $\widetilde{\delta}>0$ such that if the magnetic field strength satisfies $|H|<\widetilde{\delta}$, then there exists a traveling wave $\psi\in H^2_{\rm{loc}}(\R)$ and a propagation speed $c=c(H)$ for the reduced LLG dynamics \eqref{travelling_profile} such that $\psi-\brt\in H^2(\R)$. Moreover, the propagation speed has an expansion 
\begin{equation}\label{eq:mobility}
c = \tbet H + o(H),
\end{equation}
where the \emph{wall mobility} $\tbet > 0$ is defined as $\tbet := (M\nu)^{-1}$ with $M = \frac{1}{2}\|\brtp\|^2_{L^2} > 0$.
\end{theorem}
\begin{proof}
See the proof of Theorem 2 in \cite{CMO07}.
\end{proof}
\begin{remark}\label{rmk:1}
From the implicit function theorem we know that there exists an open neighborhood $\Omega_1\subset \R$ containing $H=0$ and another open neighborhood $\Omega_2\subset H^2(\R)\times \R$ of $(\brt,0)$ such that the mapping $H\mapsto (\psi(H),c(H))$ is differentiable and satisfies equations \eqref{travelling_profile} and $(\psi(0),c(0))=(\brt,0)$. Due to differentiability of the mapping, there exist a positive constant $\gamma$ such that 
    \begin{equation}
    \label{estpsimthe}
    \| \psi(H) - \brt \|_{H^1} + |c(H)| = H\gamma +o(H),
    \end{equation}
    for $H$ small enough (for further details, see \cite{CMO07}).
\end{remark}

\begin{remark}
\label{remP}
Once the intensity of the external magnetic field is fixed and sufficiently small, $|H| < \widetilde{\delta}$, then the propagation speed of the moving N\'eel wall's phase, $c = c(H)$ is also fixed and of order $c = O(|H|)$ as $H \to 0$. Since the wall mobility is not zero, $\widetilde{\beta} > 0$, from the implicit function theorem we know that there is a one-to-one correspondence between the parameter values $H$ and $c$. Therefore, we may use the value of $c$ as a perturbation parameter for the physical problem, instead of $H$, inasmuch as $H = O(|c|)$ as well. In the sequel we use (according to custom the nonlinear wave stability theory \cite{KaPro13}) the parameter $c$ as the perturbation parameter of the spectral problem. Also note that, from estimate \eqref{estpsimthe} we also have
\begin{equation}
\label{H1orderc}
\| \psi - \brt \|_{H^1} = O(|c|) = O(|H|).
\end{equation}
\end{remark}

We are now ready to state the main result of this paper.

\begin{theorem}[nonlinear stability of moving N\'eel walls]
\label{maintheorem}
For any given $\nu > 0$ fix $\delta \in (0, \zeta(\nu))$, where $\zeta(\nu)$ is the decay rate bound from Theorem \ref{thm:steady_result}. Let $\mathcal{J}_H \subset H^1(\R) \times L^2(\R)$ be the set of (real) initial conditions such that the Cauchy problem \eqref{reddyneq} has a global solution for a given value of the magnetic field strength, $H \in \R$. Suppose that the external magnetic field satisfies $|H| < \widetilde{\delta}$ such that $\psi = \psi(\cdot)$ is the moving N\'eel wall phase's profile from Theorem \ref{thm:CMO07} and traveling with speed $c = O(|H|)$. Then there exists some $\varep \in (0, \widetilde{\delta})$ sufficiently small such that if the initial data $(\theta_0, v_0)$ and the magnetic field strength satisfy
\[
|H| + \| \theta_0 - \psi \|_{H^1} + \| v_0 \|_{L^2} < \varep,
\]
then the solution to \eqref{reddyneq} with initial conditions $(\theta(x,0), \partial_t \theta(x,0)) = (\theta_0, v_0) \in \mathcal{J}_H$ satisfies
\[
\| \theta(\cdot, t) - \psi( \cdot + s) \|_{H^1} \leq C \exp( - \omega t),
\]
for all $t > 0$, where $s \in \R$ denotes some translation of the profile and $C, \omega > 0$ are constants depending on $H$ and $\varep$.
\end{theorem}

\begin{remark}
\label{remstab}
Just as in our previous analysis \cite{CMMP23}, Theorem \ref{maintheorem} focuses on the nonlinear stability property of the profile $\psi$: eventual small perturbations of the moving N\'eel profile, if they exist, do decay exponentially fast to a translated profile, $\psi( \cdot + s)$. The global existence of such perturbations can be proved with standard semigroup methods and using the decay estimates that we establish in this work, but we do not pursue such analysis here.
\end{remark}

\section{Preliminaries}
\label{secprel}

In this Section we gather some preliminary information which will be used later. In particular, we highlight some important features of the static N\'eel wall and its relation to the moving wall's phase.

\subsection{Properties of the static N\'eel wall's phase}

The following definitions and results can be found in the aforementioned references \cite{CMO07,CMMP23}, which the reader can consult for further information.

\begin{theorem}[Capella \emph{et al.} \cite{CMO07,CMMP23}]
\label{prev_main_res_L}
Let $\brt$ be the static N\'eel wall's phase and $\cL:L^2\to L^2$ (with domain $D(\cL)=H^2$) be the $L^2$-gradient of $\nabla \cE$ around $\brt$ defined in \eqref{eq:previousL}. Then, $\cL$ is a closed, densely defined, self-adjoint linear operator. Moreover, there exists a fixed positive constant $\Lambda_0$ such that the $L^2$-spectrum $\sigma(\cL)$ of $\cL$ satisfies  \[
\sigma(\cL)\subset \{0\}\cup [\Lambda_0,\infty).
\]
In addition,  $\lambda=0$ is a simple isolated eigenvalue with eigenspace spanned by $\brtp \in D(\cL)$.
\end{theorem}
\begin{proof}
See Lemmata 4.9 and 4.10 as well as Proposition 4.6 and Theorem 4.1 in Capella \emph{et al.} \cite{CMMP23}.
\end{proof}
\begin{definition}\label{def:Hk}
Let $L^2_\perp$ denote the $L^2$-orthogonal complement of $\brtp$. For every $k=0,1,2$, we define the following function spaces:
\[
H^k_\perp :=H^k\cap L^2_\perp.
\] 
\end{definition}

\begin{remark}
From Theorem \ref{prev_main_res_L} we know that $\lambda = 0$ is a simple eigenvalue of $\cL$ with $\ker \cL = \mathrm{span} \{ \brtp\}$. Therefore, $\cL$ is not invertible. However, when restricted to $L^2_\perp$, the operator $\cL_\perp := \cL_{|L_\perp^2} : L_\perp^2 \to L_\perp^2$ with domain $D(\cL_\perp) = H^2_\perp$, is indeed invertible because $\lambda = 0$ belongs to the resolvent of $\cL_\perp$. Hence, the inverse $\cL_\perp^{-1} : L_\perp^2 \to L_\perp^2$, with domain $D(\cL_\perp^{-1}) = L_\perp^{2}$ and range $\ran(\cL_\perp^{-1}) = H_\perp^2$, is well defined.
\end{remark}

Since $H^k$-convergence is stronger that $L^2$-convergence, it is immediate that each space $H^k_\perp$ is a Hilbert subspace (with the standard inner product) of $H^k$. The advantage of working with these spaces is that they provide the following splitting properties.

\begin{lemma}\label{lem:H1split}
Let $k\in \{0,1,2\}$. For every $u\in H^k$ there exists a unique pair $(\alpha,u_\perp) \in \C \times H^k_\perp$ such that $u = u_\perp +\alpha\brtp$. 
\end{lemma}
\begin{proof}
See the proof of Lemma 6.3 in \cite{CMMP23}.
\end{proof}

\begin{remark}
\label{remcontention}
It is to be noticed that $H^1_\perp\times L^2_\perp \subset (H^1\times L^2)_{\perp} = \ran(\cP)$, inasmuch as $U = (u,v) \in H^1_\perp\times L^2_\perp$ implies that $\langle U, \Phi_0 \rangle_{L^2 \times L^2} = \langle u, \nu \brt' \rangle_{L^2} + \langle v , \brt \rangle_{L^2} = 0$ and, consequently, $U = \cP U$.
\end{remark}

\begin{lemma}\label{lem:sesquiform}
Let $a : H^1_\perp \times H^1_\perp \to \C$ be defined by  
\begin{equation}
\label{eq:a}
a\left[ u, v\right] := \pld{\partial_x u}{\partial_x v} + \pld{s_{\theta}\cT(s_{\theta}u)}{v}-\pld{c_{\theta}u}{v}.
\end{equation}    
Then $a[\cdot,\cdot]$ is a positive sesquilinear Hermitian form that defines an inner product equivalent to $\phu{\cdot}{\cdot}$ in $H_\perp^1$, i.e. there exist two positive constants, $K, k > 0$, such that
\begin{equation}\label{eq:equivalence}
    k\nhu{u}\leq \|u\|_a\leq K\nhu{u},
\end{equation}
for all $u \in H^1_\perp$. Moreover, 
\begin{equation}\label{eq:sesquiform}
a[u,v]=\pld{\opl u}{v},\quad  u\in H_\perp^2,\ v\in H_\perp^1.
\end{equation}
\end{lemma}
\begin{proof}
See Definition 6.4, Proposition 4.6 and Lemma 6.5 in \cite{CMMP23}.
\end{proof}

The norm induced by this inner product is denoted by $\|u\|_a = \sqrt{a[u,u]}$. We will also consider the Hilbert subspace $Z=H_\perp^1\times L^2$ with the inner product $\langle \cdot,\cdot\rangle_{Z}$ given by 
\begin{equation} \label{eq:normZ}
    \langle U,V\rangle_{Z} := a[u,w] + \pld{v}{z},  
\end{equation}
for $U = (u,v)$ and $V = (w,z)$ in $Z$. It is known that $\langle \cdot,\cdot\rangle_{Z}$ and $\langle \cdot,\cdot\rangle_{H^1\times L^2}$ are equivalent inner products in $Z$ (see \cite{CMMP23}, Lemma 6.5).

\begin{remark}\label{rmk:H1xL2split}
It follows from the definition of the subspace $(H^1\times L^2)_\perp$ (see \eqref{ranofP}) that for every $\overline{U}\in (H^1\times L^2)_\perp$ then $\overline{U}=(\overline{u}, \overline{v})$ for $\overline{u}\in H^1$ and $\overline{v}\in L^2$. Moreover, by Lemma \ref{lem:H1split}, we know that $\overline{u} = u + \alpha\brtp$ and $\overline{v} = v + \beta\brtp$ for some $u\in H^1_\perp$, $v\in L^2_\perp$, and $\alpha,\beta\in \C$. With these expressions, the condition $\langle U, \Phi_0 \rangle_{L^2 \times L^2} = 0$ is rewritten as 
\[
\begin{aligned}
0=\langle U, \Phi_0 \rangle_{L^2 \times L^2}& = \big\langle \overline{u}, \nu \brtp \big\rangle_{L^2}  + \big\langle \overline{v}, \brtp \big\rangle_{L^2} \\
&= \big\langle u + \alpha \brtp, \nu \brtp \big\rangle_{L^2}  + \big\langle v + \beta \brtp, \brtp \big\rangle_{L^2} \\
&= \big( \alpha\nu + \beta \big) \big\|\brtp\big\|^2_{L^2}.
\end{aligned}
\]
In view that $\big\|\brtp\big\|_{L^2} \neq 0$ we obtain that $\alpha \nu = - \beta$. Therefore, for every $\overline{U}\in (H^1\times L^2)_\perp$ there exists a unique $U = (u,v) \in H^1_\perp\times L^2_\perp$ and $\alpha\in \C$ such that
\begin{equation}\label{eq:X_1decomposition}
\overline{U} = U + \alpha \begin{pmatrix}
    1\\ -\nu
\end{pmatrix}\brtp = \begin{pmatrix} u + \alpha \brtp \\ v - \alpha \nu \brtp \end{pmatrix}.
\end{equation}
In addition, the triangle inequality and the equivalence between the norms $\| \cdot \|_a$ and $\|\cdot\|_{H^1}$ in $H^1_\perp$ imply that
\[
\|\overline{U}\|_{H^1\times L^2}\leq \max\{1,k\}\|U\|_Z + |\alpha|\|(\brtp,-\nu \brtp)\|_{H^1\times L^2},
\]
and 
\begin{equation}\label{eq:normZ_H1xL2}
\begin{aligned}
\|U\|_Z + |\alpha|\|(\brtp,-\nu \brtp)\|_{H^1\times L^2}&\leq \max\{1,K\} \|U\|_{H^1\times L^2} + \|\alpha(\brtp,-\nu \brtp)\|_{H^1\times L^2}\\
&\leq \max\{2,1+K\}\|\overline{U}\|_{H^1\times L^2},
\end{aligned}
\end{equation}
both hold.
\end{remark}

Finally, notice that, because of Theorem \ref{thm:steady_result}, $\lambda =0$ belongs to the resolvent of the operator $\cA_\perp$ defined as the restriction of $\cA$ to $(H^1\times L^2)_\perp$. Therefore, it has a bounded and linear inverse. We end this Section by providing an explicit characterization of $\cA_\perp^{-1}$. This is a result which is not directly stated in \cite{CMO07} nor in \cite{CMMP23}, but it is an immediate consequence of Remark \ref{rmk:H1xL2split} and Theorem \ref{prev_main_res_L}. 

\begin{lemma}
\label{lem:invA}
The restriction $\cA_\perp$ of $\cA$ in $(H^1\times L^2)_\perp$ has a bounded inverse $\cA_\perp^{-1}: (H^1\times L^2)_\perp\to (H^1\times L^2)_\perp$ given by 
    \begin{equation}\label{eq:invA}
    \cA_\perp^{-1} \overline{U} = \begin{pmatrix}
        -\nu \cL_\perp^{-1} & -\cL_\perp^{-1} \\ \Id & 0
    \end{pmatrix}
    U - \alpha \begin{pmatrix}
        1/\nu \\ -1
    \end{pmatrix}\brtp,
    \end{equation}
    for all $\overline{U} \in (H^1 \times L^2)_{\perp}$ and where $U$ is given by \eqref{eq:X_1decomposition}.
\end{lemma}
\begin{proof}
    Follows from a direct computation.
\end{proof}

\subsection{Regularity of the traveling profile}

In this Section we prove that equation \eqref{travelling_profile} and the condition that $\psi\in H^2_{\rm{loc}}(\R)$ with $\psi-\brt\in H^2(\R)$ imply the smoothness of the traveling profile. Here we assume that both the static and the moving profiles, $\brt = \brt(z)$ and $\psi = \psi(z)$, respectively, are functions of the same spatial variable $z \in \R$.

\begin{lemma}
\label{lem:reg}
    Let $\psi\in H^2_{\rm{loc}}(\R)$ with $\psi-\brt\in H^2(\R)$ be a solution to \eqref{travelling_profile} for a fixed $c \in \R$, with $|c| < 1$. Then the following statements hold:
    \begin{itemize}
        \item[\rm{(a)}]\label{item:reg1} There exist two $\brt$-dependent positive constants $C_1$ and $C_2$ such that
        \[
        \begin{aligned}
        \nhu{\sin\psi-\sin\brt} &\leq C_1(\brt)\nhu{\psi-\brt},\\
        \nhu{\cos\psi-\cos\brt} &\leq C_2(\brt)\nhu{\psi-\brt}.
        \end{aligned}
        \]
        \item[\rm{(b)}]\label{item:reg2} $\psi$ is a smooth function with $\psi' \in H^k(\R)$ for any $k>0$.
    \end{itemize} 
\end{lemma}
\begin{proof}
    We first prove (a). Since $0\leq 1-\cos t\leq |t|$ for every real $t$, then by the angle-addition formula and the triangle inequality, we get 
    \[
    \begin{aligned}
    \nld{\cos \psi-\cos \brt}^2&=\nld{[1-\cos(\psi-\brt)]\cos\brt+\sin(\psi-\brt)\sin \brt}\leq 2\|\psi-\brt\|_{L^2}^2,\\
    \nld{\sin \psi-\sin \brt}^2&=\nld{[1-\cos(\psi-\brt)]\sin\brt-\sin(\psi-\brt)\cos \brt}\leq 2\|\psi-\brt\|_{L^2}^2.
    \end{aligned}
    \]
    Analogously, we obtain
    \[
    \begin{aligned}
     \nld{(\cos \psi -\cos \brt)'} &=\nld{(\psi-\brt)'\sin\psi +(\sin\psi-\sin\brt)\brt'}\\
     &\leq \nld{(\psi-\brt)'}+\|\brt'\|_{L^\infty}\nld{\psi-\brt},
    \end{aligned}
    \]
    and,
    \[
    \begin{aligned}
      \nld{(\sin\psi-\sin\brt)'} &= \nld{(\psi-\brt)'\cos \psi+(\cos \psi-\cos\brt)\brt'}\\
      &\leq \nld{(\psi-\brt)'}+2\|\brt'\|_{L^\infty}\nld{\psi-\brt}.
    \end{aligned}
    \]
    This shows (a).
    
In order to verify (b), we first notice that $\cos \psi\in H^2$. Indeed, by the reverse triangle inequality in $H^1$ and the fact that $\cos \brt\in H^1$ we first obtain
    \[
    \nhu{\cos\psi} \leq C_2(\brt)\nhu{\psi-\brt} + \nhu{\cos\brt}<\infty.
    \]
Now, notice that Sobolev's embedding in $\R$ yields $\nld{\psi'^2}^2\leq \|\psi'\|_{L^\infty}^2\nld{\psi'}^2<\infty$. Thus,
    \[
    \begin{aligned}
    \nld{(\cos \psi)''} &= \nld{(\psi-\brt)''\sin\psi +\brt''\sin\psi+ \psi'^2\cos\psi}\\
    &\leq \|(\psi-\brt)''\|_{L^2}^2+\|\brt''\|_{L^2}^2+\|\psi'^2\|_{L^2}^2\\
    &<\infty.
    \end{aligned}
    \]
 We conclude that $\cos \psi\in H^2$.

Second, a straightforward calculation of $\nabla \cE(\psi)$ implies that equation \eqref{travelling_profile} can be recast as
    \begin{equation}\label{eq:bootstraping}
    -(1-c^2)\psi''-c\nu \psi' +\sin \psi \cT \cos \psi = H\cos \psi.
    \end{equation}
    Due to the boundedness of $\cT$ (see Proposition \ref{lem:cT}) it follows that  
    \[
    \nld{\sin \psi \cT \cos \psi}\leq \nld{\cT \cos \psi}\leq C\nhu{\cos \psi}<\infty.
    \]
    Hence, Proposition \ref{lem:cT} and H\"older and Sobolev inequalities imply  
    \[
    \begin{aligned}
    \nld{(\sin \psi \cT \cos \psi)'} &\leq \nld{\psi'\cT \cos \psi} + \nhu{\cT \cos \psi}\\
     &\leq C\left[\|\psi'\|^2_{L^\infty}\nhu{\cos \psi} + \nhd{\cos\psi}\right] \\
     &<\infty.
    \end{aligned}
    \]
    
We then conclude that $\sin \psi \cT \cos \psi\in H^1$. Therefore, for $|c| < 1$, equation \eqref{eq:bootstraping} implies that $\psi''\in H^1$. A bootstrapping argument then shows that $\psi'\in H^k$ for any $k$ and the smoothness of $\psi$ follows from Sobolev's embedding. This shows the result.
\end{proof}

\section{Perturbation equations and the spectral stability problem}
\label{secperturb}

In this Section we derive the linearized equations for the perturbations of a moving N\'eel wall's phase, set up the spectral problem for the perturbations in an appropriate energy space and establish its relative boundedness with respect to the corresponding operator in the static case.

\subsection{A scalar equation for the perturbation}

Following \cite{CMO07}, we consider finite energy perturbations of the moving N\'eel wall's phase in the space $H^1$. Therefore, for any solution $\theta(x,t)$ to \eqref{reddyneq} where $\nu>0$ and $c\in \R$ are fixed, we propose the \emph{ansatz} $\theta(x,t)=u(z,t) + \psi(z)$. Here  $z=x-ct$ is the spatial coordinate in a moving frame with speed $c$, $\psi = \psi(z)$ is the traveling profile solution to \eqref{travelling_profile} and, for each fixed $t > 0$, $u(\cdot,t)\in H^1$ is the finite energy perturbation of the profile. A substitution into the differential equation in \eqref{reddyneq} yields 
\[
c^2\partial^2_z\psi - \nu c\partial_z\psi+\partial_t^2u -2c\partial_{tz}u + c^2\partial_z^2u+\nu\partial_t u-c\nu\partial_z u + \nabla \cE(u+\psi) =H\cos(u+\psi).
\]
By adding and subtracting $\nabla\cE(\psi)+H\cos \psi$ and using the profile equation \eqref{travelling_profile}, we arrive at
\begin{equation}\label{eq:perturbation}
\partial_t^2u -2c\partial_{tz}u + c^2\partial_z^2u+\nu\partial_t u-c\nu\partial_z u + \nabla \cE(u+\psi)- \nabla\cE(\psi) = H(\cos(u+\psi)-\cos \psi).
\end{equation}
Equation \eqref{eq:perturbation} can be written as a first order system by introducing the auxiliary variable  $v=\partial_t u$. Indeed, after a straightforward computation one obtains
\begin{equation}\label{eq:perturbation_sys}
\partial_t\begin{pmatrix}
    u \\ v
\end{pmatrix} = \begin{pmatrix}
    0 & \Id \\ -\cL_c & 2c\partial_z-\nu\Id
\end{pmatrix}\begin{pmatrix}
    u \\ v
\end{pmatrix}+\begin{pmatrix}
    0 \\ \cN(u)
\end{pmatrix},    
\end{equation}
where $\cL_c : H^1 \to L^2$ is a linear, densely defined operator with domain $D(\cL)=H^2$ and $\cN(u) : L^2\to L^2$ comprises the nonlinear terms, again with $D(\cN) = H^2$. They are explicitly given by
\begin{equation}\label{eq:Lc}
    \begin{aligned}
        \cL_c &:= \ocL_c +c^2\partial^2_z-c\nu\partial_z+H s_\psi \, \Id,\\
        \cN(u) &:= -\nabla\cE(\psi+u)+\nabla\cE(\psi)+\ocL_c u + H(\cos(\psi+u)-\cos \psi+s_\psi u ),
    \end{aligned}
\end{equation}
where the operator $\ocL_c:L^2\to L^2$ is the $L^2$-linearization of $\nabla \cE$ around $\psi$ with domain $D(\ocL_c)=H^2$. A direct calculation implies that 
\begin{equation}
        \label{defL0}
        \ocL_c u := - \partial^2_z u + s_ \psi\cT (s_\psi u) - c_\psi u, \qquad u \in D(\cL) = H^2.
    \end{equation}
Upon substitution into \eqref{eq:Lc},
\begin{equation}
\label{eq:L}
\left\{
\arraycolsep=1pt\def\arraystretch{1.5}
\begin{array}{l}
    \cL_c : H^1 \to L^2,\\
    D(\cL_c) = H^2,\\
    \cL_c u = -(1-c^2)\partial^2_z u + s_\psi\cT (s_\psi u) -c\nu\partial_z u - (c_\psi-Hs_\psi) u.
    \end{array}
    \right.
\end{equation}
The latter is the explicit expression of the linearized operator around the moving N\'eel wall's phase for a fixed value of the magnetization $H$ (which, in turns, fixes the velocity $c$) and whose spectral properties determine its stability.

\subsection{The spectral problem}

According to custom \cite{KaPro13,San02} we start our stability analysis by disregarding the nonlinear terms and by focusing on the associated linear spectral problem. To that end, notice that the linearization of equation \eqref{eq:perturbation} around the traveling N\'eel profile reads
\begin{equation}
 \label{linequ}
\partial_t^2u -2c\partial_{tz}u +\nu\partial_t u +\cL_c u = 0.
\end{equation}
Last equation is equivalent to the following linear system,
\begin{equation}
 \label{linsyst}
\partial_t\begin{pmatrix}
    u \\ v
\end{pmatrix} = \begin{pmatrix}
    0 & \Id \\ -\cL_c & 2c\partial_z-\nu\Id
\end{pmatrix}\begin{pmatrix}
    u \\ v
\end{pmatrix}.    
\end{equation}
We specialize equation \eqref{linequ} to perturbations of the form $e^{\lambda t}u(z)$, where $\lambda \in \C$ is the spectral parameter. This  substitution yields the following (non-standard) spectral problem
\begin{equation}
 \label{spectralu}
(\lambda^2- 2c\lambda\partial_{z} +\nu\lambda) u + \cL u = 0.
\end{equation}

The latter is a spectral equation which is quadratic in $\lambda$, also known as a \textit{quadratic operator pencil} (cf. Markus \cite{Ma88}). The transformation $v = \lambda u$ (the spectral equivalent of the change of variables $v = \partial_t u$) is used to write equation \eqref{spectralu} as a genuine eigenvalue problem of the form
\begin{equation}
 \label{evproblem}
  \cA_c \begin{pmatrix}
         u \\ v
        \end{pmatrix}:=
        \begin{pmatrix}
         0 & \Id \\ -\cL_c & 2c\partial_z-\nu\Id
        \end{pmatrix} \begin{pmatrix}
         u \\ v
        \end{pmatrix} =
\lambda \begin{pmatrix}
         u \\ v
        \end{pmatrix},
\end{equation}
but now posed on a base space defined as a Cartesian product. The matrix operator $\cA_c$ is often called the companion matrix to the quadratic spectral pencil (see \cite{BrJoK14,KHKT13} for more information). Notice that equation \eqref{evproblem} is the spectral equation associated to the linear system \eqref{linsyst}. We shall refer to both \eqref{spectralu} and \eqref{evproblem} as the spectral problem making no distinction between them.

Therefore, we are interested in the spectral properties of the following \emph{block matrix operator} (for each fixed value of $c$),
\[
 \cA_c : H^1 \times L^2 \to H^1 \times L^2, 
\]
regarded as a linear, densely defined operator in the Hilbert space $H^1 \times L^2$ with domain $D(\cA_c) := H^2 \times H^1$ and defined by  
\begin{equation}
\label{eq:A_c}
\begin{pmatrix}
    u\\ v
\end{pmatrix}\; \mapsto \;  \begin{pmatrix}
    0 & \Id \\ -\cL_c & 2c\partial_z-\nu\Id
    \end{pmatrix} \begin{pmatrix}
    u \\ v
\end{pmatrix},
\end{equation}
for all $(u,v) \in H^2 \times H^1$.

\begin{remark}
\label{rmk:a_zero}
From the expressions \eqref{eq:A_c} and \eqref{eq:L} we notice that in the case when $c=0$, the operators $\cA_0$ and $\cL_0$ coincide with the operators $\cA$ and $\cL$, respectively, for the static N\'eel wall phase (see Theorem \ref{thm:steady_result}). From now on and for the sake of brevity, we simply write $\cA$ and $\cL$ to denote the ``static'' operators.
\end{remark}

\begin{remark}
\label{remimportant}
At this point we call the reader's attention to the following fundamental property of the operator $\cA_c$: once the intensity of the external magnetic field has been set, $|H| \ll 1$, then the speed of the moving N\'eel wall gets fixed and it is small as well, $|c| = O(|H|)$. As a result, if we examine the expression of the linearized operator $\cL_c$ around the moving profile, we notice that it can be recast as
\[
\begin{aligned}
\cL_c u &= \cL u + c^2 \partial_z^2 u + s_\psi(z) \cT (s_\psi(z)u) - c \nu \partial_z u  + \big( c_{\brt} (z) - c_\psi(z) + H s_\psi(z) \big) u \\ 
&\quad - s_{\brt(z)} \cT (s_{\brt(z)}u),
\end{aligned}
\]
for all $u \in H^2$, where
\[
\cL u = - \partial_z^2 u + s_{\brt}(z) \cT(s_{\brt}(z) u) - c_{\brt}(z) u,
\]
is the linearized operator around the static N\'eel profile, $\brt = \brt(z)$, evaluated at the Galilean variable of translation, $z = x-ct$. Whence, the block operator $\cA_c$ can be written as 
\[
\cA_c = \cA + \cB_c,
\]
where the block perturbation operator
\[
\left\{ 
\arraycolsep=1pt\def\arraystretch{1.5}
\begin{array}{l}
\cB_c : D(\cB_c) \subset H^1 \times L^2 \to H^1 \times L^2,\\
D(\cB_c) = H^2 \times H^1,
\end{array}
\right.
\]
is defined as
\begin{equation}
\label{defBc}
\cB_c := \begin{pmatrix}
0 & 0 \\ -\cS & 2c \partial_z
\end{pmatrix},
\end{equation}
and the non-local operator $\cS : D(\cS) = H^2 \subset H^1 \to L^2$ is determined via the relation
\begin{equation}
\label{eq:S}
\begin{aligned}
    \cS u &:= c^2\partial_z^2u- c\nu\partial_z u + s_\psi(z) \cT(s_\psi(z) u)- s_{\brt}(z) \cT(s_{\brt}(z) u) \\
    &\quad + \big(c_{\brt}(z) -c_\psi(z) + H s_\psi(z) \big)u,
    \end{aligned}
    \end{equation}
    for all $u \in H^2$. In what follows we verify that this is a relatively bounded perturbation of the static operator $\cA$.
\end{remark}

\subsection{Relative boundedness}

We recall that given two linear operators $\cP$ and $\cS$ on a Banach space $X$, we say that $\cS$ is {\em $\cP$-bounded} if $D(\cP)\subset D(\cS)$ and there exists two non-negative constants $a$ and $b$ such that
\begin{equation}\label{eq:P_bound}
\|\cS u\| \leq a\|u\|+b\|\cP u\|, \qquad \forall \, u\in D(\cP).     
\end{equation}
The greatest lower bound $b_0$ of all admissible constants $b$ is called the {\em $\cP$-bound} of $\cS$.

\begin{lemma}
\label{lem:growth_bounds}
The operator $\cB_c = \cA_c-\cA$ is $\cA$-bounded in $H^1\times L^2$ with coeficients $a(c)$ and $b(c)$ that are continuous and increasing functions of $c$ and both tend to zero as $c$ does.
\end{lemma}

An immediate consequence of Lemma \ref{lem:growth_bounds} is the following result.
\begin{corollary}\label{cor:growth_bounds}
    The operator $\cA_c$ is $\cA$-bounded in $H^1\times L^2$ with coefficients $\tilde a(c)$ and $\tilde b(c)$ that are continuous and increasing functions of $c$ such that $\tilde a(c)$ and $1-\tilde b(c)$ tend to zero as $c$ does.
\end{corollary}

\begin{proof} [Proof of Lemma \ref{lem:growth_bounds}]
First notice that for each $U = (u,v) \in D(\cA_c) = H^2 \times H^1$ we have
    \begin{equation} 
    \label{eq:norms_rel_bound}
    \begin{aligned}
        \|\cB_c U\|_{H^1\times L^2} &= \nld{\cS u-2c\partial_z v}, \\  
        \|\cA U\|^2_{H^1\times L^2} &= \nhu{v}^2+\nld{\cL u+\nu v}^2. 
        \end{aligned}
    \end{equation} 
    
    Second, we bound $\|\cB_c U\|_{H^1\times L^2}$ by a sum of terms depending on powers of $H$ (or, equivalently, on powers of $c$). Substitution of \eqref{eq:S} into the expression of $\cB_cU$, as well as addition and subtraction of  $c^2[s_{\brt}\cT(s_{\brt}u)-c_{\brt}u]$, yields
    \[
    \begin{aligned}
    \cB_c U=(0,-&c^2[-\partial_z^2u+s_{\brt}\cT(s_{\brt}u)-c_{\brt}u+\nu v] -2c\partial_z v +c^2\nu v 
   +c^2[s_{\brt}\cT(s_{\brt}u)-c_{\brt}u] + \\
   & +Hs_\psi u- c\nu\partial_z u + s_\psi \cT(s_\psi u)- s_{\brt} \cT(s_{\brt} u)+(c_{\brt}-c_\psi)u).
    \end{aligned}
    \]
    Using \eqref{eq:previousL} we get 
    \begin{equation}
    \label{la44}
    \begin{aligned}
        \|\cB_c U\|_{H^1\times L^2} &\leq \nld{c^2(\cL u+ \nu v)+2c\partial_z v -c^2\nu v} + \\
        & \quad + c^2\nld{s_{\brt}\cT(s_{\brt}u)-c_{\brt}u} +\\ 
        & \quad +\nld{Hs_\psi u-c\nu \partial_z u}+ \nld{s_\psi \cT (s_\psi u)-s_{\brt} \cT (s_{\brt} u)}+ \\
        & \quad + \nld{(c_\psi-c_{\brt})u}.
    \end{aligned} 
    \end{equation} 
    
    Third, we estimate each term of the right hand side of last inequality. The first term is directly bounded by Young's inequality, since 
    \[
    \begin{aligned}
    \nld{c^2(\cL u+ \nu v)+2c\partial_z v +c^2\nu v}^2 &\leq 3(c^4\nld{\cL u+ \nu v}^2 + 4c^2\nld{\partial_z v}^2 +c^4\nu^2\nld{v}^2)\\
    &\leq 3\max\{c^4,4c^2,c^4\nu^2\}(\nld{\cL u+ \nu v}^2 +\nhu{v}^2).
    \end{aligned}
    \]
    Since we are interested in the behavior of the operators for small values of $|c|$ and in the limit when $c \to 0$, we may assume that $|c|<1$. Therefore, we get the estimate
    \begin{equation}
    \label{eq:aux1}
    \begin{aligned}
    \nld{c^2(\cL u+ \nu v)+2c\partial_z v +c^2\nu v} &\leq 2\max\{2|c|,c^2\nu\} \|\cA U\|_{H^1\times L^2}.
    \end{aligned}
    \end{equation}
    
    For the second term on the right hand side of \eqref{la44}, we use the triangle inequality and the boundedness of $\cT:H^{1}\to L^2$ (see Proposition \ref{lem:cT}) to obtain
    \begin{equation}\label{eq:aux2}
    \begin{aligned}
    c^2\nld{s_{\brt}\cT(s_{\brt}u)-c_{\brt}u} &\leq c^2[\nld{s_{\brt}\cT(s_{\brt}u)}+\nld{c_{\brt}u}]\\ & \leq c^2[C\nhu{(s_{\brt}u)}+\nld{c_{\brt}u}]\\ &\leq C(\brt)c^2\nhu{u}.
    \end{aligned}
    \end{equation}

    In the same fashion, we use \eqref{eq:mobility} in order to estimate the third term in \eqref{la44}, yielding
    \begin{equation}\label{eq:aux3}
    \begin{aligned}
    \nld{Hs_\psi u-c\nu \partial_z u} &\leq |H|\left(\nld{s_\psi u}+\frac{c\nu}{H}\nld{ \partial_z u}\right) \\
    &\leq |H|\left[\nld{u}+ \frac{2}{\|\brt'\|_{L^2}^2+o(1)} \nld{\partial_z u}\right]\\
    &\leq |H|\left[\nld{u}+ \frac{2}{\|\brt'\|_{L^2}^2+o(1)} \nld{\partial_z u}\right] \\
    &=|H|C(\brt)\nhu{u}.
    \end{aligned}
    \end{equation}
    
   In order to estimate the fourth term in the right hand side of \eqref{la44}, we add and subtract the term $s_\psi\cT(s_{\brt}u)$ and use the boundedness of $\cT$ (Proposition \ref{lem:cT}). The result is
    \begin{equation*}
    \begin{aligned}
    \nld{s_\psi \cT (s_\psi u)-s_{\brt} \cT (s_{\brt} u)} 
    &\leq \nld{s_\psi \cT ((s_\psi-s_{\brt}) u)}+\nld{(s_\psi-s_{\brt}) \cT (s_{\brt} u)}\\
    &\leq C\nhu{(s_\psi-s_{\brt}) u}+\|(s_\psi-s_{\brt})\|_{L^\infty}\nhu{ s_{\brt} u}.
    \end{aligned}
    \end{equation*}
    Since $H^1$ is a Banach algebra, its norm sub-distributes any product. This fact and Sobolev's inequality imply the estimate
    \begin{equation}
    \label{eq:aux4} 
    \begin{aligned}
    \nld{s_\psi \cT (s_\psi u)-s_{\brt} \cT (s_{\brt} u)} &\leq C\nhu{s_\psi-s_{\brt}}(\nhu{u}+\nhu{ s_{\brt} u}) \\ 
    &\leq C(\brt)\nhu{\psi-\brt}\nhu{u},
    \end{aligned}
    \end{equation}
    where the last inequality follows from Lemma \ref{lem:reg} (a).
    
    Finally, we estimate the last term in \eqref{la44}. By H\"older and Sobolev inequalities, we clearly have
    \[
    \nld{(c_\psi-c_{\brt})u} \leq \nld{c_\psi-c_{\brt}}\|u\|_{L^\infty}\leq \nld{c_\psi-c_{\brt}}\nhu{u}.
    \]
    We symmetrize previous relation by adding and subtracting $\cos \psi \cT \cos\brt$. Hence, together with Proposition \ref{lem:cT}, one obtains
    \begin{equation}\label{eq:aux5}
    \begin{aligned}
    \nld{(c_\psi-c_{\brt})u} 
    &\leq \nld{\cos \psi \cT (\cos\psi-\cos\brt)+ (\cos\psi-\cos\brt)\cT \cos \brt }\nhu{u}\\
    &\leq \left [ \nld{\cT (\cos\psi-\cos\brt)} + \nld{(\cos\psi-\cos\brt)\cT \cos \brt   }\right]\nhu{u}\\
    &\leq C(\brt)\nhu{\cos\psi-\cos\brt}\nhu{u}\\    
    &\leq C(\brt)\nhu{\psi-\brt}\nhu{u}.    
    \end{aligned}
    \end{equation}
    Combining estimates \eqref{eq:aux1} thru \eqref{eq:aux5}, we arrive at
    \begin{equation}
    \label{Bcestc}
    \begin{aligned}
    \|\cB_c U\|_{H^1\times L^2} &\leq  2\max\{2|c|,c^2\nu\} \|\cA U\|_{H^1\times L^2} \\ & \quad  + C(\brt)(c^2 + H + 2 \| \psi-\brt \|_{H^1}) \| u \|_{H^1}.
    \end{aligned}
    \end{equation}
    Due to Remark \ref{rmk:1} and Theorem \ref{thm:CMO07}, the result now follows.
\end{proof}
A direct consequence of Lemma \ref{lem:growth_bounds} is the closedness of the block operator $\cA_c$ for $c$ small enough.
\begin{corollary}\label{cor:closedness}
    The block operator $\cA_c$ is closed in $H^1\times L^2$ for every $c$  such that $b(c)<1$.
\end{corollary}
\begin{proof}
From Theorem IV.1.1 in Kato \cite{Kat80}, p. 190,  we easily conclude that $\cA_c$ is a closed operator because $\cA_c = \cA+\cB_c$, where $\cB_c$ is $\cA$-bounded with $\cA$-bound less that $b(c)<1$ and $\cA$ is a closed operator (see Lemma 5.4 in \cite{CMMP23}).
\end{proof}

\begin{remark}
Estimate \eqref{Bcestc} and Remark \ref{remP} imply that, in fact, $\cA_c$ is an $O(|c|)$-perturbation of the static operator $\cA$ (indeed, substitute \eqref{H1orderc} and $H = O(|c|)$ into \eqref{Bcestc} to conclude). This is the main observation that motivates the strategy of analyzing the spectrum/resolvent of $\cA_c$ as a perturbation of the spectrum/resolvent of $\cA$.
\end{remark}

\section{Resolvent-type estimates and spectral stability}
\label{secresests}

In this Section we establish the spectral stability of the family of operators $\cA_c$ for sufficiently small parameter values of $|c|$. To that end, we apply perturbation theory of linear operators and the spectral stability of the linearized block operator $\cA$ around the static N\'eel wall (see Theorem \ref{thm:steady_result} and the companion paper \cite{CMMP23}). The strategy of proof is based on establishing resolvent-type estimates which allow to locate the resolvent set of the perturbed operator in terms of the resolvent set of the operator around the static wall. 

We start by showing that the translation eigenvalue of $\cA_c$ is isolated for all $|c|$ sufficiently small.

\begin{lemma}
\label{lem:zeroev}
    Let the operator $\cA_c$ be as in \eqref{evproblem} and let $\cA$ and $\zeta(\nu)>0$ be as in Lemma \ref{thm:steady_result}. Fix $\delta \in(0,\zeta(\nu))$. Then there exists $c_\delta>0$ such that the square  $\Gamma$ in the complex plane with length side $2\delta$ and center at $\lambda = 0$, more precisely (see Figure \ref{fig1contour} below), 
    \begin{equation}
   \label{defofcontourGamma} 
    \Gamma := \big\{ z \in \C \, : \, z = \zeta \pm i \delta, \, \zeta \in [-\delta,\delta]\big\} \cup \big\{ z \in \C \, : \, z = \pm \delta + i \zeta, \, \zeta \in [-\delta,\delta]\big\},
    \end{equation}
     belongs to $\rho(\cA_c)$ and such that
\[
\sigma(\cA_c) \cap \intconv (\Gamma) = \{ 0 \}, \qquad \text{for all} \;\; |c| \leq c_\delta,
\]     
where $\intconv (\Gamma)$ denotes the interior of the convex hull of $\Gamma$. Moreover, $\lambda=0$ is an isolated simple eigenvalue of the block operator $\cA_c$ with eigenspace spanned by 
\begin{equation}
\label{defofThetac}
\Theta_c:=(\partial_z \psi,0) \in D(\cA_c) = H^2 \times H^1.
\end{equation}    
\end{lemma}
\begin{proof}
Let $\delta \in (0,\zeta(\nu))$ be fixed. First, we prove the existence of $c_\delta>0$ such that $\Gamma$ isolates $\lambda = 0$ from the rest of the spectrum of $\cA_c$ for all $|c| \leq c_\delta$. Since $\lambda = 0$ is an isolated simple eigenvalue of $\cA$ and the rest of the spectrum satisfies $\Re \lambda<-\zeta(\nu)$ (see Theorem \ref{thm:steady_result}), we know that the closed simple curve $\Gamma\subset \rho(\cA)$ divides $\sigma(\cA)$. Moreover, it induces a partition on the working space $H^1\times L^2 = M'(\cA)\oplus M''(\cA)$ where $M'(\cA)$ and $M''(\cA)$ are the images of $H^1\times L^2$ under the linear operators $\cP$ and $\Id-\cP$, respectively,  with 
\[
\cP:=\Id- \int_{\Gamma} (\cA-\lambda \Id)^{-1} \, d\lambda.
\]
Because $\|(\cA-\lambda)^{-1}\|$ and $ \|\cA(\cA-\lambda)^{-1}\|$  are analytic functions in $\lambda$, we conclude that both functions are bounded on $\Gamma$.
Also, from Lemma \ref{lem:growth_bounds} we  know that $\cA_c = \cA+\cB_c$ and $\cB_c$ is $\cA$-bounded with constants $a(c)$ and $b(c)$ that are continuous and increasing in $c$ with $a(0)=b(0)=0$ but they do not depend on $\lambda$. These facts imply the existence of $c_\delta>0$  small enough such that 
    \[
 a(c_\delta) \sup_{\lambda\in\Gamma}\|(\cA-\lambda)^{-1}\|+ b(c_\delta) \sup_{\lambda\in\Gamma}\|\cA(\cA-\lambda)^{-1}\|<1.
    \]
Due to Lemma \ref{lem:growth_bounds}, properties of the supremum imply that 
\[
    \sup_{\lambda\in\Gamma} \left( a(c)\|(\cA-\lambda)^{-1}\|+ b(c) \|\cA(\cA-\lambda)^{-1}\|\right)< 1,
\]
for every $c$ with $|c|\leq c_\delta$. Therefore, Theorems IV.3.18 and IV.3.16 in Kato \cite{Kat80}, pp. 212 -- 214, imply that $\Gamma\subset \rho(\cA_c)$ and that it splits $\sigma(\cA_c)$ into two sets. Moreover, the induced partition on the working space $H^1\times L^2 = M'(\cA_c)\oplus M''(\cA_c)$ satisfies $\dim M'(\cA) = \dim M'(\cA_c)$ and $\dim M''(\cA) = \dim M''(\cA_c)$. 

Second, we verify that $\lambda = 0$ is an isolated simple eigenvalue of $\cA_c$. Indeed, we know that $(\partial_z \psi,0) \in \ker \cA_c$; this follows by noticing that $\cA_c (\partial_z \psi,0) =(0,\cL_c \partial_z \psi)$, but a direct differentiation of \eqref{travelling_profile} implies $\cL_c \partial_z \psi=0$, for any value of $c$. If $|c|\leq c_\delta$ holds then we conclude that the algebraic multiplicity of $\lambda=0$ is equal to one because $\dim M'(\cA) = \dim M'(\cA_c)$. The lemma is proved. 
\end{proof}

Lemma \ref{lem:zeroev} implies that $\Gamma$ separates the origin from the rest of the spectrum, but it still does not guarantee the spectral stability nor the existence of an spectral gap in the following sense: that for each $0 < |c| \ll 1$ there exists $\zeta_c > 0$ such that $\sigma(\cA_c) \backslash \{ 0 \}$ is strictly contained in $\{\Re \lambda < - \zeta_c < 0\}$. The following results are devoted to prove these properties.

\begin{lemma}
\label{lem:res_equations}
 Let $\lambda\in \C$ be fixed and different from zero. Suppose that $U=(u,v) $ and $F=(f,g) $ belong to $H^1_\perp\times L^2_\perp$. 
\begin{itemize}
\item[\rm{(a)}] \label{resolvent_ineq1} If $U, F$ satisfy 
\begin{equation}\label{resolvent_equa1}
(\lambda-\cA)U = F, 
\end{equation}
then $u,\, v,\,f,$ and $g$ satisfy the following estimate
\begin{equation}\label{eq:bound0}
\Big| \lambda^*\|u\|_a^2 + (\lambda+\nu)\nld{v}^2 \Big| \leq \|U\|_Z\|F\|_Z.
\end{equation} 
\item[\rm{(b)}] \label{resolvent_ineq2} If $U, F$ satisfy 
\begin{equation}\label{resolvent_equa2}
(\cA-\lambda)\cA_\perp^{-1}U = F, 
\end{equation}
then $u,\, v,\,f,$ and $g$ satisfy the following estimates,
\begin{subequations}
\begin{empheq}{align}
\Big| \lambda^*\|u\|_a^2 + (\nu+\lambda)\nld{v}^2 \Big| &\leq C_1(|\lambda|\|u\|_a + |\lambda +\nu|\nld{v})\|F\|_{Z}, \label{eq:bound1}\\
\Big| \|u\|_a - |\nu+\lambda|\Lambda_0^{-1/2}\nld{v} \Big| &\leq C_2 \|F\|_Z, \label{eq:bound2}
\end{empheq}
\end{subequations}
for some constants $C_1, C_2 > 0$ and where $\Lambda_0 > 0$ is the spectral bound for $\cL$ from Theorem \ref{prev_main_res_L}.
\end{itemize}
\end{lemma}
\begin{remark}
\label{allinversesok}
Notice that the operator $\cA_\perp^{-1}$ is well defined on $H^1_\perp \times L^2_\perp \subset \ran (\cP)$ (see Remark \ref{remcontention}). Thus we may use the representation formula from Lemma \ref{lem:invA}. Observe that $\cL_\perp^{-1}$ is also well defined on $L_\perp^2$.
\end{remark}
\begin{proof}[Proof of Lemma \ref{lem:res_equations}]
First, we assume that $U, F \in H^1_\perp\times L^2_\perp$ satisfy equation \eqref{resolvent_equa1}, which is equivalent to the system
\[
\begin{aligned}
\lambda u - v &= f,\\
\opl u +(\lambda+\nu)v &= g.
\end{aligned}
\]
Now, we apply $a[u,\cdot]$ to the first equation and $\pld{\cdot}{v}$ to the second equation, yielding
\[
\begin{aligned}
a[u,\lambda u] - a[u,v] &= a[u,f], \\
\pld{\cL u}{v} +(\lambda+\nu)\nld{v}^2 &= \pld{g}{v}.
\end{aligned}
\]
The properties of the sesquilinear form $a[\cdot,\cdot]$ (see, for instance, Lemma~\ref{lem:sesquiform} and \eqref{eq:sesquiform}), imply that  
\[
\begin{aligned}
\lambda^* a[u,u] - a[u,v] &= a[u,f], \\
a[u,v] +(\lambda+\nu)\nld{v}^2 &= \pld{g}{v}.
\end{aligned}
\]
By adding both equations we obtain
\[
\lambda^*\|u\|_a^2 + (\lambda+\nu)\nld{v}^2 = a[u,f]+\pld{v}{g}.
\]
Hence, relation \eqref{eq:bound0} follows by noticing that $\langle U,F \rangle_Z =a[u,f]+\pld{v}{g}$ and by applying Cauchy-Schwarz inequality to the inner product in $Z$. This shows (a).

Let us prove (b). First, notice that relation \eqref{resolvent_equa2} implies that 
\begin{subequations}
\label{eq:system}
\begin{empheq}[left=\empheqlbrace]{align}
     u + \lambda \cL_\perp^{-1}(\nu u+v) &= f, \label{eq:resa}\\
     v-\lambda u &= g, \label{eq:resb}
\end{empheq}
\end{subequations}
inasmuch as $\nu u+v \in L^2_\perp = D(\cL_\perp^{-1})$. By adding and subtracting $\lambda\nu^{-1}v$, equation \eqref{eq:resb} is recast as 
\[
 (\nu+\lambda)v-\lambda(\nu u+v) = \nu g \in L^2_\perp.
\]
Apply $\cL_\perp^{-1}$ to last equation and add equation \eqref{eq:resa} in order to obtain
\begin{equation}
\label{eq:resc}
u+(\nu+\lambda)\cL_\perp^{-1}v =f+ \nu \cL_\perp^{-1}g.
\end{equation}
Also, from equation \eqref{eq:resb} we know that $v-g=\lambda u \in H^1_\perp$. Therefore, 
\[
\lambda^*a[u,u]+(\nu + \lambda)a[\cL_\perp^{-1}v ,v-g]=\lambda^*a[f+\nu \cL_\perp^{-1}g,u]. 
\]
Since $\cL_\perp^{-1}:L^2_\perp \to H^2_\perp$, from \eqref{eq:sesquiform} we get
\begin{equation}\label{eq:resd0}
\lambda^*\|u\|_a^2+(\nu + \lambda)\pld{v}{v-g}=\lambda^*a[f+\nu \cL_\perp^{-1}g,u].
\end{equation}

Now, if we proceed differently and distribute the $L^2$-inner product in equation \eqref{eq:resd0} then we arrive at
\begin{equation}\label{eq:prime1}
\lambda^*\|u\|_a^2 + (\nu+\lambda)\nld{v}^2=\lambda^*a[f 
+\nu\cL_\perp^{-1}g
,u] +(\nu+\lambda)\pld{v}{g}.
\end{equation}
Use identity \eqref{eq:sesquiform} in equation \eqref{eq:prime1}, take its modulus and apply the Cauchy-Schwarz inequality twice. The result is
\[
\begin{aligned}
|\lambda^*\|u\|_a^2 + (\nu+\lambda)\nld{v}^2|&=|\lambda^*a[f,u] 
+\nu\lambda^*\pld{g}{u} +(\nu+\lambda)\pld{v}{g}|,\\
&\leq |\lambda|(\|f\|_a|\|u\|_a+\nu\nld{g}\nld{u}) + |\lambda +\nu|\nld{g}\nld{v},\\
&\leq \max\{1,\nu\}(\|f\|_a+\nld{g})(|\lambda|\|u\|_a +|\lambda +\nu|\nld{v}).
\end{aligned}
\]
Therefore, relation \eqref{eq:bound1} follows. 

Let us now substitute $v-g = \lambda u$ into \eqref{eq:resd0}.  This yields,
\begin{equation}\label{eq:res_aux}
\|u\|_a^2 + (\nu+\lambda)\pld{v}{u}=a[f 
+\nu\cL_\perp^{-1}g
,u].
\end{equation}
Thus, if we take the $L^2$-inner product between $v$ and \eqref{eq:resc} we then obtain
\begin{equation*}
\pld{v}{u} + (\nu+\lambda)^*\pld{v}{\cL_\perp^{-1}v} = \pld{v}{f} + \nu^*\pld{v}{\cL_\perp^{-1}g}.
\end{equation*}
We use this last equation to simplify \eqref{eq:res_aux}, yielding 
\begin{equation}\label{eq:prime2}
\|u\|_a^2 - |\nu+\lambda|^2\pld{v}{\cL_\perp^{-1}v} = a[f,u] +\nu\pld{g}{u}-(\nu+\lambda)\pld{v}{f+\nu\cL_\perp^{-1}g}. 
\end{equation}

Therefore relation \eqref{eq:bound2} follows from \eqref{eq:prime2} by noticing that $\pld{v}{\cL_\perp^{-1}v}\leq \Lambda_0^{-1}\nld{v}^2$ for every $v\in H^1_\perp$ (see Theorem \ref{prev_main_res_L} above).  Indeed, the latter lower bound for the left hand side of \eqref{eq:prime2} implies that
\begin{equation}\label{eq:resz1}
\|u\|_a^2 - |\nu+\lambda|^2\Lambda_0^{-1}\nld{v}^2\leq |\|u\|_a^2 - |\nu+\lambda|^2\pld{v}{\cL_\perp^{-1}v} |.
\end{equation}
Moreover, the right hand side of \eqref{eq:prime2} can be bounded above by
\[
\begin{aligned}
|a[f,u] +\nu\pld{g}{u}-(\nu+\lambda)\pld{v}{f+\nu\cL_\perp^{-1}g}|&\leq  \|f\|_a\|u\|_a +\nu\nld{g}\nld{u}+\\
&\; \; +|\nu+\lambda|\nld{v}(\nld{f}+\nu\nld{\cL_\perp^{-1}g}).
\end{aligned}
\]
Since $\|\cdot \|_a$ and $\nhu{\cdot}$ are equivalent in $H^1_\perp$ (see \eqref{eq:equivalence} at Lemma \ref{lem:sesquiform}) we know that there exists $k>0$ such that $\nld{u}\leq k^{-1}\|u\|_a$ for every $u\in H^1_\perp$. Hence,
\begin{equation}\label{eq:resz2}
\begin{aligned}
|a[f,u] &+\nu\pld{g}{u} - (\nu+\lambda)\pld{v}{f+\nu\cL_\perp^{-1}g}|
\leq \\
&\leq (\|f\|_a + \nu k^{-1}\nld{g})\|u\|_a + |\nu+\lambda|\Lambda_0^{-1/2}\nld{v}(\Lambda_0^{1/2}k^{-1}\|f\|_a+\nu\nld{g})\\
&\leq C(\|u\|_a+|\nu+\lambda|\Lambda_0^{-1/2}\nld{v})(\|f\|_a + \nld{g}),
\end{aligned}
\end{equation}
where $C = \max\{1,\nu k^{-1},\nu,\Lambda_0^{1/2}k^{-1}\} > 0$. The proof is complete by noticing that equation \eqref{eq:bound2} is obtained upon a substitution of \eqref{eq:resz1} and \eqref{eq:resz2} in \eqref{eq:prime2}. This shows (b) and the lemma is now proved.
\end{proof}

\begin{remark}\label{rmk:res_inequalities}
Notice that unless $\|u\|_a=0$ and $\nld{v}=0$, the left hand term of identity \eqref{eq:prime1} vanishes if either, (i) $\lambda\in \R$ and $\nu\nld{v}^2 = -\lambda (\|u\|^2_a+\nld{v}^2)$, or (ii) $\|u\|_a =\nld{v} $ and $\Re \lambda = -\nu/2$. This observation will be helpful later on. 
\end{remark}

Next, we prove the spectral stability of the moving profile with a positive spectral gap. The strategy of proof is based on establishing resolvent-type estimates for the operator around the static N\'eel wall (namely, to estimate $\|(\cA-\lambda \Id)^{-1} \|$ and $\|\cA(\cA-\lambda \Id)^{-1} \|$) which yield, in turn, the location of the spectrum of $\cA_c$ for $|c|$ small enough upon application of the standard perturbation theory for linear operators.

\begin{theorem}[resolvent-type estimates]
\label{lem:spectral_gap}
Let $\nu>0$ be fixed, $\cA_c$ be as in \eqref{evproblem} and  $\cA$, $\zeta(\nu)>0$ be as in Lemma \ref{thm:steady_result}. Also, let $\delta \in(0,\zeta(\nu))$ be fixed and let $\Gamma\subset \C$ denote the square with length side $2\delta$ and center at $\lambda = 0$ as in Lemma \ref{lem:zeroev}. If $c_\delta>0$ is small enough such that Lemma \ref{lem:zeroev} holds then both $\|(\cA-\lambda \Id)^{-1} \|$ and $\|\cA(\cA-\lambda \Id)^{-1} \|$ are uniformly bounded for $\lambda$ in the set
\[
G:=\{\lambda\in \C \,|\, \Re \lambda>-\delta\}\setminus \intconv(\Gamma).
\]
\end{theorem}
\begin{proof}
First, we prove that $\|(\cA-\lambda)^{-1}\|$ is uniformly bounded. Let $\widetilde{U}, \widetilde{F}\in H^1\times L^2$  and let $\cP:H^1\times L^2\to (H^1\times L^2)_\perp$ be the projector operator that commutes with $\cA$ and which is defined in \eqref{defofP}.  By the definition of $\cP$ it follows that $\Id-\cP$ also commutes with $\cA$ and $\cA(\Id-\cP)$ that is the null operator. Then, 
\[
\begin{aligned}
(\cA-\lambda)\widetilde{U} &= (\cA -\lambda)(\cP \widetilde{U} +(\Id-\cP)\widetilde{U})\\
&= (\cA -\lambda)\cP \widetilde{U}+(\cA -\lambda)(\Id-\cP)\widetilde{U}\\
&=(\cA -\lambda)\cP \widetilde{U}-\lambda(\Id-\cP)\widetilde{U},
\end{aligned}
\]
and the relation $(\cA-\lambda)\widetilde{U} = \widetilde{F}$  implies that
\[
(\cA -\lambda)\cP \widetilde{U}-\lambda(\Id-\cP)\widetilde{U} = \cP \widetilde{F} + (\Id-\cP)\widetilde{F}.
\]
Last equation splits into the following two equations,
\[
(\cA -\lambda)\cP \widetilde{U} = \cP \widetilde{F}, \quad \mbox{and }-\lambda(\Id-\cP)\widetilde{U} = (\Id-\cP)\widetilde{F}.
\]
Therefore,
\begin{equation}
\label{eq:aux_resxx}
\|\widetilde{U}\|_{H^1\times L^2}\leq \|\cP\widetilde{U}\|_{H^1\times L^2}+\|(\Id-\cP)\widetilde{U}\|_{H^1\times L^2} \leq \left[\|\cP(\cA-\lambda)^{-1}\cP\|+\frac{1}{|\lambda|}\right]\|\widetilde{F}\|_{H^1\times L^2},
\end{equation}
because $\cP$ and $\Id-\cP$ are projector operators. Estimate \eqref{eq:aux_resxx} implies that the proof is complete provided that we verify  that $(\cA-\lambda)^{-1}$, regarded as an operator in $(H^1\times L^2)_\perp$, is uniformly bounded on $G$. To that end, we must prove that there exists a uniform bound $C$ such that, if $\overline{U},\overline{F}\in (H^1\times L^2)_\perp$ satisfy 
\begin{equation}\label{eq:res_xx}
(\cA-\lambda)\overline{U} = \overline{F},
\end{equation}
then they also satisfy the following estimate,
\[
\|\overline{U}\|_{H^1\times L^2}\leq C\|\overline{F}\|_{H^1\times L^2}. 
\]
Once again, by Remark \ref{rmk:H1xL2split}, we know that $\overline{U} = U+\alpha(1,-\nu)  \brtp$ and $\overline{F} = F+\beta(1,-\nu)  \brtp$ for some $U,F\in H^1_\perp\times L^2_\perp$ and $\alpha,\beta\in \C$. Thus, equation \eqref{eq:res_xx} implies that
\[
(\cA-\lambda)U = F, \quad \text{and} \quad -\alpha(\nu + \lambda) = \beta. 
\]
Then, $\alpha = -\beta/(\nu+\lambda)$ and by Lemma \ref{lem:res_equations} (a),
\[
|\lambda^*\|u\|_a^2 + (\lambda+\nu)\nld{v}^2| = \|U\|_Z\|F\|_Z,
\]
for $U =(u,v) $. Substituting $\|u\|_a = \|U\|_Z \cos \phi$ and $\nld{v} = \|U\|_Z \sin \phi$ for $\phi\in [0,\pi/2]$, we have 
\[
\|U\|_Z\leq C(\lambda,\phi) \|F\|_Z,
\]  
where 
\[
C(\lambda,\phi)=\frac{1}{\sqrt{(\Re\lambda+\nu\sin^2\phi)^2 + (\Im \lambda)^2\cos^22\phi}}\leq \frac{\sqrt{\delta^2 + \tfrac{\nu^2}{4}}}{\delta(\tfrac{\nu}{2}-\delta)},
\]
thanks to Lemma \ref{lem:aux0_mod}. In addition, Remark \ref{rmk:H1xL2split} and equation \eqref{eq:normZ_H1xL2} yield
\[
\begin{aligned}
\|\bar{U}\|_{H^1\times L^2}&\leq |\alpha|\|(1,-\nu)\partial_x\brt\|_{H^1\times L^2} + \|U\|_{H^1\times L^2} \\
&\leq \frac{1}{|\nu+\lambda|}\|\beta(1,-\nu)\partial_x\brt\|_{H^1\times L^2} + 
\frac{K\sqrt{\delta^2 + \tfrac{\nu^2}{4}}}{k\delta(\tfrac{\nu}{2}-\delta)}
\|F\|_{H^1\times L^2}\\
&\leq  \left[\frac{1}{|\nu+\lambda|} + \frac{K\sqrt{\delta^2 + \tfrac{\nu^2}{4}}}{k\delta(\tfrac{\nu}{2}-\delta)}
\right](\|\beta(1,-\nu)\partial_x\brt\|_{H^1\times L^2}+\|F\|_{H^1\times L^2})\\
&\leq  \sqrt{2}\left[\frac{1}{|\nu+\lambda|} + \frac{K\sqrt{\delta^2 + \tfrac{\nu^2}{4}}}{k\delta(\tfrac{\nu}{2}-\delta)}
\right]\|\overline{F}\|_{H^1\times L^2}.
\end{aligned}
\] 
Upon substitution of last inequality in \eqref{eq:aux_resxx} we conclude that
\[
\|\widetilde{U}\|_{H^1\times L^2}\leq \left[\frac{\sqrt{2}}{|\nu+\lambda|} + \frac{K\sqrt{\delta^2 + \tfrac{\nu^2}{4}}}{k\delta(\tfrac{\nu}{2}-\delta)}
+\frac{1}{|\lambda|}\right]\|\widetilde{F}\|_{H^1\times L^2}.
\] 
Since the right hand side of the last equation is uniformly bounded on $G$, we obtain that $\|(\cA-\lambda)^{-1}\|$ is uniformly bounded on $G$.

Finally, we prove that $\|\cA(\cA-\lambda)^{-1}\|$ is uniformly bounded on $G$. Due to Theorem \ref{thm:steady_result}, we know that $\cA$ is the infinitesimal generator of a $C_0$-semigroup of quasicontractions and, by virtue of Feller-Miyadera-Phillips' theorem (see, e.g., Engel and Nagel \cite{EN06}, p. 69), we know there exists $w \in \R$ such that for every $\lambda \in \res{\cA}$ with $\Re \lambda>w$ there holds
\begin{equation}\label{eq:FMP}
\|(\cA-\lambda \Id)^{-1}\|_{H^1\times L^2}\leq \frac{1}{\Re \lambda -w}.
\end{equation}
We then choose a positive constant $M_1>w$ and divide the set $G$ into three subsets as follows (see Figure \ref{fig1setG}):
\[
\begin{aligned}
G_1 &= \{\lambda\in G \, : \, \Re \lambda > M_1, \;  |\Im \lambda|\leq\delta\},\\
G_2 &= \{\lambda\in G \, : \, |\Im \lambda|>\delta \}\cup\{\delta(t \pm i) \, : \, t\in (-1,1)\},\\
G_3 &= G\setminus(G_1\cup G_2).
\end{aligned}
\]

\begin{figure}[t]
\begin{center}
\subfigure[Isolation of the origin thanks to the spectral gap.]{\label{fig1contour}\includegraphics[scale=0.6]{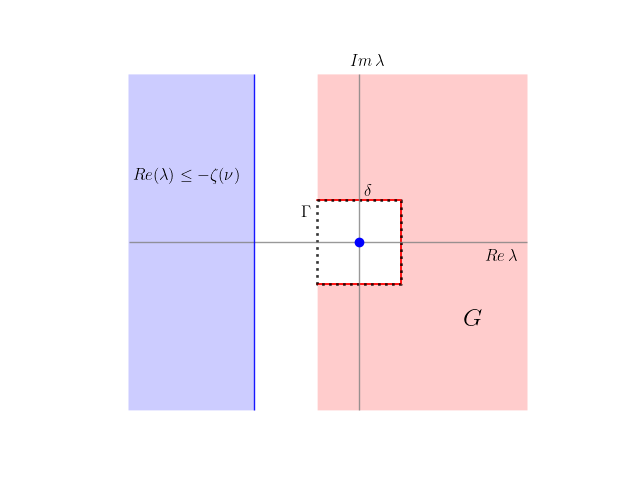}}
\subfigure[Partition of $G$.]{\label{fig1setG}\includegraphics[scale=0.6]{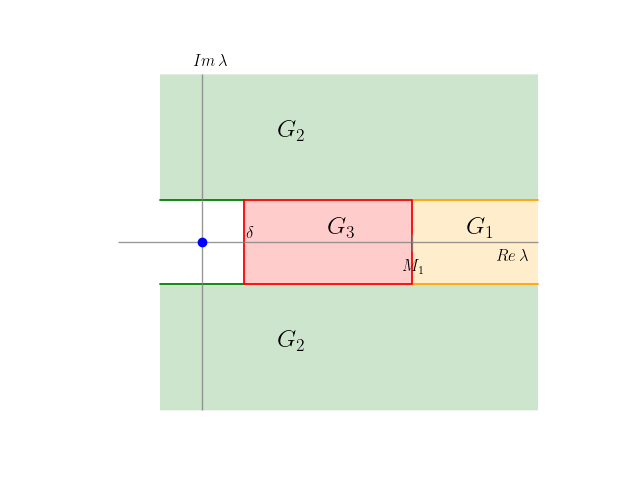}}
\caption{\small{Panel (a) shows the set $\{\lambda\in\C\,|\,\Re \lambda<-\zeta(\nu)\}\cup\{0\}$ which contains $\sigma(\cA)$ (blue region),  the curve $\Gamma$ that splits $\sigma(\cA_c)$ (dotted black curve), and the set $G$ contained in $\rho(\cA_c)$ (red zone). Panel (b) displays the suggested partition of $G$, where  $G_3$ is compact (color online). }
}
\label{fig1}
\end{center}
\end{figure}

Notice that $\|\cA(\cA-\lambda)^{-1}\|$ is uniformly bounded in $G_1$ and $G_3$. Indeed, since $\lambda\mapsto \|\cA(\cA-\lambda)^{-1}\|$ is analytic and $G_3$ is compact, the uniform bound in $G_3$ follows from standard results. To verify the assertion for $G_1$ use the identity  $\cA(\cA-\lambda)^{-1} = \Id+\lambda(\cA-\lambda)^{-1}$ and relation \eqref{eq:FMP} to arrive at
\begin{equation}\label{eq:boundG1}
\|\cA(\cA-\lambda)^{-1}\| \leq 1+|\lambda|\|(\cA-\lambda)^{-1}\| \leq 1+\frac{|\lambda|}{\Re \lambda-w},
\end{equation}
for every $\lambda\in G$. Notice that the right hand side of \eqref{eq:boundG1} is uniformly bounded on $G_1$ and so is $\|\cA(\cA-\lambda)^{-1}\|$.

%
%

It remains to get a uniform bound on $G_2$. By Theorem \ref{thm:steady_result}, we know that there exists a projector operator $\cP:H^1\times L^2\to (H^1\times L^2)_\perp$ that commutes with $\cA$. Moreover, $\cP$ and $(\cA-\lambda)^{-1}$ also commute for every $\lambda\in \rho(\cA)$ (see Kato \cite{Kat80}, p. 173). Then, for every $\widetilde{F}\in H^1\times L^2$, there holds that $\widetilde{F} = \cP \widetilde{F} + (\Id -\cP)\widetilde{F}$ and 
\[
\cA(\cA-\lambda)^{-1}\widetilde{F} = \cA\cP(\cA-\lambda)^{-1}\widetilde{F} =\cA\cP^2(\cA-\lambda)^{-1}\widetilde{F}  = \cP\cA(\cA-\lambda)^{-1}\cP \widetilde{F}. 
\]

Thus, last equation implies that the operator norm of $\cA(\cA-\lambda)^{-1}$ in $H^1\times L^2$ is equal to the norm of the same operator in $(H^1\times L^2)_{\perp}$. In addition, $\cA$ is invertible as a restricted operator on $(H^1\times L^2)_\perp$ (see Lemma \ref{lem:invA}). It is straightforward to verify that the required uniform boundedness of $\|\cA(\cA-\lambda)^{-1}\|$ on $G_2$, regarded as an operator from $(H^1\times L^2)_\perp$ into itself, is readily implied by showing the existence of a constant $C$ independent of $\lambda$ such that for every $\overline{U},\overline{F}\in (H^1\times L^2)_\perp$ satisfying
    \begin{equation}
    \label{eq:vector_sys}
    (\cA-\lambda)\cA_\perp^{-1}\overline{U} = \overline{F},
    \end{equation}
    they also satisfy the estimate 
    \[
    \|\overline{U}\|_{H^1\times L^2}\leq C\|\overline{F}\|_{H^1\times L^2}.
    \] 
    Moreover, Remark \ref{rmk:H1xL2split} implies that $\overline{U},\overline{F}\in (H^1\times L^2)_\perp$ is equivalent to 
    \[
 \overline{U} = U
 + \alpha \begin{pmatrix}
    1\\ -\nu
\end{pmatrix}\brtp, \qquad \mbox{and }\qquad     
\overline{F} = 
F+ \beta \begin{pmatrix}
    1\\ -\nu
\end{pmatrix}\brtp,
    \]
for some $U,F\in H^1_\perp\times L^2_\perp$. Under these assumptions, system \eqref{eq:vector_sys} reduces to
\begin{equation}\label{eq:res_system}
(\cA-\Id)\cA_\perp^{-1}U = F,
\end{equation}
\begin{equation}\label{eq:resz}  
\alpha(1+\lambda\nu^{-1})-\beta =0.   
\end{equation}
From \eqref{eq:resz} it follows that $\alpha=\beta/(1+\lambda\nu^{-1})$, and from Lemma \ref{lem:res_equations} (b) we have that the entries of $U$ and $F$ satisfy conditions \eqref{eq:bound1} and \eqref{eq:bound2}. 

In addition, we assume that $\|u\|_a = \|U\|_Z\cos \phi$ and $\nld{v} = \|U\|_Z\sin \phi$ for $\phi\in[0,\pi/2]$ due to  the expression for $\|U\|_Z$ (see equation \eqref{eq:normZ}). Hence, relations \eqref{eq:bound1} and \eqref{eq:bound2} yield
\begin{subequations}
\begin{empheq}{align}
     \|U\|_Z|\lambda^*\cos^2\phi + (\lambda+\nu)\sin^2\phi|&\leq C_1 (|\lambda|\cos \phi +|\lambda+\nu|\sin\phi)\|F\|_Z , \label{eq:sim_bound1}\\
 	\|U\|_Z|\cos \phi-|\lambda+\nu|\Lambda_0^{-1/2}\sin \phi|& \leq C_2\|F\|_Z. \label{eq:sim_bound2}
\end{empheq}
\end{subequations}
Thus, equations \eqref{eq:sim_bound1} and \eqref{eq:sim_bound2} imply that
\[
\|U\|_Z\leq 2C M(\phi, \lambda)\|F\|_Z,
\]
for some positive constant $C$ and with
\[
M(\phi, \lambda)= \min\left\{\frac{1}{|\cos \phi-|\lambda+\nu|\Lambda_0^{-1/2}\sin \phi|}, \frac{|\lambda| +|\lambda+\nu|}{|\lambda^*\cos^2\phi + (\lambda+\nu)\sin^2\phi|}
\right\}.
\]

Notice that $M(\phi,\lambda)$ is well-defined for every $(\phi,\lambda)\in [0,\pi/2]\times G_2$ since $|\lambda^*\cos^2\phi + (\lambda+\nu)\sin^2\phi|\neq 0$ (see Remark \ref{rmk:res_inequalities}). Moreover by Lemma \ref{lem:aux3_mod} we conclude that $M(\phi,\lambda)$ is uniformly bounded in $[0,\pi/2]\times G_2$. This implies, in turn, that $\|\cA(\cA-\lambda)^{-1}\|$ is uniformly bounded on $G_2$. Therefore, it is uniformly bounded on $G$ and the proof is complete.
\end{proof}

\begin{theorem}[spectral stability with spectral gap]
\label{theospectralstab}
Under the assumptions of Theorem \ref{lem:spectral_gap} there exists $c_\delta'\in (0,c_\delta)$ such that
\begin{equation}
\label{eq:spectrum_location}
\sigma (\cA_c) \subset \{0\}\cup \{\lambda\in \C \,|\,\Re \lambda\leq -\delta\},
\end{equation}
for every speed $c$ such that $|c|<c_\delta'$. 
\end{theorem}
\begin{proof}
First we prove that, if both $\|(\cA-\lambda \Id)^{-1} \|$ and $\|\cA(\cA-\lambda \Id)^{-1} \|$ are uniformly bounded for $\lambda \in G$, then \eqref{eq:spectrum_location} holds.	Since the hypotheses of Lemma \ref{lem:zeroev} hold, we know that $\Gamma$ divides $\sigma(\cA_c)$ into two sets: $\{0 \}$, which is contained in contained in $\intconv(\Gamma)$, and $\sigma(\cA_c) \backslash \{ 0 \}$, contained in the exterior of $\conv(\Gamma)$.

It suffices to consider $\lambda \in \sigma(\cA_c) \backslash \{ 0 \}$. For the part of the spectrum contained in the exterior of $\conv(\Gamma)$ we use the fact that the term $a(c_\delta)\|(\cA-\lambda \Id)^{-1} \| + b(c_\delta)\|\cA(\cA-\lambda \Id)^{-1} \|$ is uniformly bounded for $\lambda \in\{\lambda\in \C \,|\, \Re \lambda>\-\delta\}\setminus \intconv (\Gamma)$ and, consequently, Lemma \ref{lem:growth_bounds} implies the existence of $c_\delta'\in (0,c_\delta)$ such that 
\[
a(c)\|(\cA-\lambda \Id)^{-1} \| + b(c)\|\cA(\cA-\lambda \Id)^{-1} \|<1,
\]
for every $c$ with $|c|<c_\delta'$. Therefore, we are now able to apply standard perturbation theory of linear operators, in particular, Theorem IV.3.17 in Kato \cite{Kat80}, p. 214, and to conclude that the whole set $G$ belongs to $\rho(\cA_c)$ for all $|c|<c_\delta'$. The theorem is proved.
\end{proof}

\section{Generation of the semigroup and decay estimates}
\label{secsemiggen}

In this Section, we show that for sufficiently small values of $c$, the operator $\cA_c : H^1\times L^2 \to H^1\times L^2$ is the infinitesimal generator of a $C_0$- semigroup which, when restricted to the appropriate subspace, it becomes exponentially decaying. To accomplish this task, we assume that for every fixed parameter value $\nu > 0$, $\zeta(\nu)>0$ is defined as in Lemma \ref{thm:steady_result}, $\delta \in(0,\zeta(\nu))$ is also kept fixed, and  $c_\delta, c_\delta'$ are the positive constants such that Theorem \ref{theospectralstab} holds. 
 
 \subsection{The adjoint operator}
 
 It is well-known that if $\lambda \in \C$ is an eigenvalue of a closed operator $\cA : \cD \subset H \to H$ then $\lambda^*$ is an eigenvalue of $\cA^*$ (formal adjoint) with the same geometric and algebraic multiplicities (see Kato \cite{Kat80}, Remark III.6.23, p. 184). In our case, since $H^1\times L^2$ is a reflexive Hilbert space, and  $\cA_c : H^1\times L^2\to H^1\times L^2$, with domain $\cD=H^2\times H^1$, is a closed operator for $c$ small enough (see Lemma \ref{lem:growth_bounds} and Corollary \ref{cor:closedness}) it has a formal adjoint which is also closed and densely defined. Moreover, $\cA_c^{**} = \cA_c$ (cf. \cite{Kat80}, Theorem III.5.29, pg. 168). Upon these observations we immediately have the following result.
\begin{lemma}
 \label{lempre1}
$\lambda = 0$ is an isolated, simple eigenvalue of $\cA_c^* : (H^1\times L^2)^* \to (H^1\times L^2)^*$, and there exists an eigenfunction $\Psi_c \in (H^1\times L^2)^*$ such that
\[
 \cA_c^* \Psi_c = 0.
\]
\end{lemma}

It is to be observed that, by Riesz representation theorem, there exists $\widetilde{\Psi}_c \in H^1 \times L^2$ such that
\[
\Psi_c(U) = \langle U, \widetilde{\Psi}_c \rangle_{H^1 \times L^2},
\] 
for all $U \in H^1 \times L^2$. Let us denote the inner product of both eigenfunctions as
\[
R_c := \langle \Theta_c, \widetilde{\Psi}_c \rangle_{H^1 \times L^2}.
\]
For each $c$ small, $\widetilde{\Psi}_c$ (respectively, $\Theta_c$) is the eigenfunction associated to the simple eigenvalue $\lambda = 0$ of the operator $\cA_c^*$ (respectively, $\cA_c$).

\begin{lemma}
 \label{Thetanotzero}
$R_c \neq 0$.
\end{lemma}
\begin{proof}
By contradiction, suppose that $R_c = 0$. This implies that $\langle \Theta_c, \widetilde{\Psi}_c \rangle_{H^1 \times L^2} = 0$ or, equivalently, that
\[
\widetilde{\Psi}_c \in \big( \ker \cA_c \big)^{\perp} = \ran(\cA_c^*).
\]
Therefore, there exists $\Phi \in D(\cA_c^*)$, $\Phi \neq 0$ (because $\widetilde{\Psi}_c \neq 0$), such that $\cA_c^* \Phi = \widetilde{\Psi}_c$. But this is a contradiction with the fact that $\lambda = 0$ is a simple eigenvalue of $\cA_c^*$.
\end{proof}

Now, we define $\widetilde{X} \subset H^1\times L^2$ as the range of the spectral projection,
\begin{equation}
\label{eq:projector}
\left\{
\arraycolsep=1pt\def\arraystretch{1.5}
\begin{array}{l}
\cP_c : H^1 \times L^2 \to H^1 \times L^2,\\
\cP_c U := U - R_c^{-1}\langle U, \widetilde{\Psi}_c \rangle_{H^1 \times L^2}\,\Theta_c,\\
\widetilde{X} := \ran(\cP_c).
\end{array}
\right.
\end{equation}
\begin{remark}\label{rmk:x}
Notice that $\widetilde{X} = \ran(\cP_c)$ is the range of a bounded linear operator and, consequently, it is a (closed) Hilbert subspace of $H^1\times L^2$.
\end{remark}
 In this fashion we project out the eigenspace spanned by the single eigenfunction $\Theta_c$. We will see in the next Section that for $|c|$ small enough, the block operator $\cA_c$ is a generator of a $C_0$-semigroup of quasicontractions such that, outside this eigenspace, decays exponentially. 
 
 \subsection{Generation of the semigroup}
 
In this Section we apply Lumer-Phillips's theorem and the fact that the operator $\cA_c$ is an $O(c)$-perturbation of the block operator $\cA$, in order to show that $\cA_c$ is the infinitesimal generator of a $C_0$-semigroup for each $c$ small enough.

\begin{theorem}
The operator $\cA_c : H^1\times L^2 \to H^1\times L^2$, with domain $D(\cA_c) = H^2\times H^1$, is the infinitesimal generator of a $C_0$-semigroup of quasicontractions $\{e^{t\cA_c}\}_{t\geq 0}$ for every speed $c$ such that $|c|<c_\delta'$.  Moreover, for each $U \in H^2\times H^1$,
\begin{equation}\label{eq:semigroup_property}
\frac{d}{dt} \big( e^{t\cA_c}U \big) = e^{t\cA_c} \cA_c U = \cA_c(e^{t\cA_c}U).
\end{equation}
\end{theorem}
\begin{proof}
In order to apply Lumer-Phillips's theorem we need to verify that
\begin{itemize}
\item[(i)] there exists $\omega_0 \in \R$ such that
\[
 \Re \langle \cA_c U,U\rangle_{H^1\times L^2} \leq \omega_0 \|U\|^2_{H^1\times L^2},
\]
for each $U \in D(\cA_c) = H^2\times H^1$;  and that
\item[(ii)] there exists $\tau_0 > \omega_0$ such that $\cA - \tau_0$ is onto;
\end{itemize}
(see, e.g., \cite{ReRo04}, Theorem 12.22, p. 407). From Theorem \ref{theospectralstab}  we already know that if $|c| < c_\delta'$, with $c_\delta' > 0$ sufficiently small, then the whole positive real semi-axis belongs to the resolvent of $\cA_c$, that is,
\[
(0, \infty) \subset \rho(\cA_c)
\]
(see \eqref{eq:spectrum_location}). Hence, $\cA_c - \tau$ is onto for any $\tau > 0$ and (ii) holds for any $\tau_0 > \omega_0$, where $\omega_0$ is the constant found in (i). Therefore, it suffices to prove (i) for some $\omega_0 > 0$.
 
As in Lemma \ref{lem:growth_bounds}, we know that $\cA_c$ is $\cA$-bounded with constants $a(c)$ and $b(c)$, which are continuous and increasing functions of $c$. Moreover, $\cA_c = \cA+\cB_c$, where $\cB_c (u,v)  = (0,-\cS u + 2c\partial_z v)$ and $\cS$ is the nonlocal operator defined in \eqref{eq:S}. Then, for every $U \in D(\cA_c) = H^2\times H^1$, the Cauchy-Schwarz inequality implies that 
\[
\begin{aligned}
\Re \langle\cA_c U,U\rangle _{H^1\times L^2}&=\Re \langle\cA U,U\rangle_{H^1\times L^2}+\Re \langle\cB_c U,U\rangle_{H^1\times L^2}\\
&\leq\Re \langle\cA U,U\rangle_{H^1\times L^2}+ \|\cB_c U\|_{H^1\times L^2}\|U\|_{H^1\times L^2}.
\end{aligned}
\]
Also, we know that $\cB_c$ is $\cA$-bounded and $\cA$ is the generator of a $C_0$-semigroup of quasicontractions. Thus,
\[
\|\cB_cU\|_{H^1\times L^2}\leq a(c)\|U\|_{H^1\times L^2} +  b(c)\|\cA U\|_{H^1\times L^2},
\]
and
\[
\Re \langle\cA_c U,U\rangle _{H^1\times L^2}\leq \omega \langle U,U\rangle _{H^1\times L^2},
\]
for some $\omega\in \R$ and every $U\in H^2\times H^1$. These two observations yield
\[
\Re \langle\cA_c U,U\rangle _{H^1\times L^2}
\leq(\omega+a(c)+b(c)) \| U\|^2_{H^1\times L^2}.
\]
Therefore (i) holds by setting $\omega_0\geq \omega+a(c)+b(c)$. Finally, identity \eqref{eq:semigroup_property} holds by standard theory of $C_0$-semigroups (see, e.g., Pazy \cite{Pa83}). 
\end{proof}

\begin{lemma}
\label{lemma:Pcommutes}
Let $\cP_c$ be the projection operator given in \eqref{eq:projector}.
Then,  
$e^{t \cA_c}  \cP_c = \cP_c e^{t \cA_c}$ for  all $t \geq 0$.
\end{lemma}

\begin{proof}
In the reflexive Banach space $H^1\times L^2$, weak and weak$^*$ topologies coincide. Therefore, the family of dual operators $\{ (e^{t \cA_c})^*\}_{t\geq 0}$, consisting of all the corresponding formal adjoints in $(H^1\times L^2)^*$, is also a $C_0$-semigroup (cf. \cite{EN00}, p. 44) with infinitesimal generator $\cA_c^*$ (see Corollary 10.6 in \cite{Pa83}). Hence, 
\begin{equation}
\label{eq:deulexp}
    (e^{t \cA_c})^* = e^{t \cA_c^*}.
\end{equation}
Now, let  $U\in H^1\times L^2$. Then from the definition of $\cP_c$ we have 
\[ 
\cP_c \big( e^{t \cA_c} U \big) = e^{t \cA_c} U - R_c^{-1} \langle e^{t \cA_c}U, \widetilde{\Psi}_c \rangle_{H^1 \times L^2} \Theta_c. 
\]
Since $\Theta_c$ and $\widetilde{\Psi}_c$ belong to the kernels of their respective generators, equation \eqref{eq:deulexp} and the standard semigroup theory (cf. \cite{Pa83,EN00}) yield
\[
e^{t \cA_c} \Theta_c = \Theta_c, \quad \text{and} \quad \big(e^{t \cA_c}\big)^*\widetilde{\Psi}_c = e^{t \cA_c^*} \widetilde{\Psi}_c = \widetilde{\Psi}_c.
\]
This implies that
\[
\langle e^{t \cA_c}U, \widetilde{\Psi}_c \rangle_{H^1 \times L^2} \Theta_c = \langle U, \big(e^{t \cA_c}\big)^*\widetilde{\Psi}_c \rangle_{H^1 \times L^2} \Theta_c = \langle U, \widetilde{\Psi}_c \rangle_{H^1 \times L^2} e^{t \cA_c} \Theta_c,
\]
verifying, in this fashion, the identity
\[
\cP_c \big( e^{t \cA_c} U \big) = e^{t \cA_c} \Big( U - R_c^{-1} \langle U, \widetilde{\Psi}_c \rangle_{H^1 \times L^2} \Theta_c \Big) = e^{t \cA_c} \cP_c U,
\]
for all $U\in H^1\times L^2$, as claimed.
\end{proof}

\subsection{Exponential decay}

We denote the restriction of $\cA_c$ in $\widetilde{X}$ by $\widetilde{\cA}_c: \widetilde{X} \to \widetilde{X}$ with domain
\[
 \widetilde{D} := \{U \in D(\cA_c) \cap \widetilde{X} \, : \, \cA_c U \in \widetilde{X} \},
\]
and assignation rule given by
\[
 \widetilde{\cA}_c U := \cA_c U, \qquad U \in \widetilde{D}.
\]
The following lemma shows that $\lambda = 0$ does not belong to $\sigma(\widetilde{\cA}_c)$ and that this operator is the infinitesimal generator of an exponentially decaying $C_0$-semigroup.
\begin{lemma} 
\label{elcoro}
Let $\widetilde{X}$ be the range of $\cP_c$ and let $\widetilde{\cA}_c$ denote the restriction of $\cA_c$ on $\widetilde{X}$ as above. Then, 
\begin{itemize}
\item[\rm{(a)}] \label{item:1_restricted_op} The block operator $\widetilde{\cA}_c$ is a closed, densely defined operator on the Hilbert space $\widetilde{X}$.
\item[\rm{(b)}] \label{item:2_restricted_op} The spectrum of $\widetilde{\cA}_c$ satisfies that
\[
\sigma(\widetilde{\cA}_c)\subset \{\lambda\in\C\,|\, \Re \lambda \leq -\delta\}
\]
\item[\rm{(c)}] \label{item:3_restricted_op} The family of operators $\{e^{t\widetilde{\cA}_c}\}_{t\geq 0}$, defined as 
\[
 e^{t \widetilde{\cA}_c}U := e^{t \cA} U, 
\]
for $U \in \widetilde{X}$ and  $t \geq 0$, is a $C_0$-semigroup of quasicontractions in the Hilbert space $(H^1\times L^2)_\perp$ with infinitesimal generator $\tilde\cA_c$.
\item[\rm{(d)}] \label{item:4_restricted_op} There exists uniform constants $M \geq 1$ and $\omega > 0$, such that
\begin{equation}
\label{lindecay}
\|e^{t \widetilde{\cA}_c} U\|_{H^1\times L^2} \leq M e^{-\omega t} \|U\|_{H^1\times L^2},
\end{equation}
for all $t \geq 0$ and every $U \in (H^1\times L^2)_{\perp}$.
\end{itemize}
\end{lemma}

\begin{proof}
Since $\widetilde{\cA}_c$ is a densely defined operator then its closedness is inherited from the closedness of $\cA_c$ and the closedness of the subspace $\widetilde{X}$. This shows (a).

We now prove (b). By the spectral decomposition theorem we know that $\sigma(\cA_\perp)\subset \sigma(\cA)$ (cf. \cite{Kat80}), but $0\notin \sigma(\cA_\perp)$ because $\cP \Theta = 0$ and $\Theta\neq 0$  so that $\Theta \notin \widetilde{X}$. Therefore, by Theorem \ref{lem:spectral_gap}, we have that $\sigma (\widetilde{\cA}_c) \subset \sigma (\cA_c)\setminus\{0\} \subset \{\lambda \in \C \, : \, \Re  \lambda \leq - \delta \}$ as claimed.

By Lemma \ref{lemma:Pcommutes} and Remark \ref{rmk:x}, we conclude that $\widetilde{X}$  is an $e^{t \cA}$-invariant closed Hilbert subspace of $H^1 \times L^2$. Therefore, property (c) readily follows from a direct application of classical results from semigroup theory (see Section 2.3 of \cite{EN00}, p. 61). Thus, $\widetilde{\cA}_c$ is the infinitesimal generator of the ``restricted'' semigroup $\{e^{t \cA_\perp}\}_{t\geq 0}$.

Finally, (d) follows by an application of Gearhart-Pr\"uss theorem and since the operator $\widetilde{\cA}_c$ is spectrally stable (see (b)), it remains to prove that  
\begin{equation}
    \label{lemF6}
 \sup_{\Re \lambda > 0} \| (\lambda-\widetilde{\cA}_c)^{-1} \|_{\widetilde{X} \to \widetilde{X}} <  \infty,
\end{equation}
for every $\lambda\in \{\lambda\in\C \|\ \Re \lambda >0\}$ holds. Notice that from (b) we know that the mapping $\lambda\mapsto \|(\lambda-\widetilde{\cA}_c)^{-1}\|_{\widetilde{X} \to \widetilde{X}}$ is continuous, hence it is uniformly bounded in the compact set $K:=\{\lambda \in\C\,|\, 0\leq \Re\ \lambda\leq\delta, \, \wedge \, |\Im \lambda |\leq \delta\}$. By Theorem \ref{lem:spectral_gap} we know that $\|(\lambda-\cA_c)^{-1}\|$ is uniformly bounded in $\{\lambda\in\C \, : \, \Re \lambda >0\}\setminus K$; hence, the proof is complete by noticing that 
\[
\|(\lambda-\widetilde{\cA}_c)^{-1}\|_{\widetilde{X} \to \widetilde{X}} \leq \|(\lambda- \cA_c)^{-1}\|_{X \to X}
\]
since $\widetilde{X}\subset X = H^1\times L^2$ and $\widetilde{\cA}_c = \cA_c$ on $\widetilde{D}$ by definition.
\end{proof}

We are now ready to prove the main result of the paper: the nonlinear stability of moving N\'eel wall's phase for $|c|$ sufficiently small.

\section{Nonlinear (orbital) stability of moving N\'eel walls} 
\label{sec:nonlinear_stability}

In this Section, we establish the nonlinear (orbital) stability of the moving N\'eel wall, that is, the decay in time of solutions $\theta(x,t)$ to the Cauchy problem \eqref{reddyneq}, provided they exist, for initial conditions close to the traveling N\'eel wall profile $\psi(z)$. This is accomplished by recasting equation \eqref{reddyneq} as a new equation that underlies the traveling N\'eel wall profile $\psi(z)$ as an stationary point. 

By using the moving frame variables $(z,t)$, with $z = x-ct$, the differential equation for the N\'eel wall phase $\theta(x,t) = \widetilde{\theta}(z,t)$ induces the following differential equation on $\widetilde{\theta}$:
\begin{equation}\label{eq:movil}
\partial_t^2\widetilde{\theta} -2c\partial_{tz}\widetilde{\theta} + c^2\partial_z^2\widetilde{\theta}+\nu\partial_t \widetilde{\theta}-c\nu\partial_z \widetilde{\theta} + \nabla \cE(\widetilde{\theta}) =H\cos(\widetilde{\theta}).
\end{equation}
To ease the notation, we drop out all the tilded variables in equation \eqref{eq:movil} from now on.
Notice that the equation \eqref{eq:movil} is also invariant under translations in the variable $z$. Moreover, $\theta = \psi(z)$ is a stationary solution since $\psi$ is solution to \eqref{travelling_profile}. By defining the auxiliary variable $\varphi=\partial_t \theta$, equation \eqref{eq:movil} is recast as the following vector evolution equation 
\begin{equation}
\label{eq:FOnlpert}
\left\{
\begin{aligned}
\partial_t W &= F(W),  & z &\in \R, \, t > 0, \\ 
W(z,0)&=W_0(z), & z &\in \R,  
\end{aligned}\right.
\end{equation}
where $W = (\theta,\varphi)$,  $W_0 = (\theta_0, v_0)$ is the initial datum and 
\begin{equation}\label{eq:F}
F(W) = 
\begin{pmatrix}
\varphi\\
2c\partial_{z}\varphi - c^2\partial_z^2\theta-\nu\varphi+c\nu\partial_z \theta - \nabla \cE(\theta) +H\cos(\theta)
\end{pmatrix}.
\end{equation}
Due to the existence of the stationary state $\psi$ and the translation invariance of equation \eqref{eq:movil}, the term $F$ satisfies that for every $s\in \R$
\begin{equation}\label{F-invariance} 
F(\phi(s)) = 0, \quad \mbox{where}\quad \phi(s) =\begin{pmatrix}\psi(\cdot+s)\\0\end{pmatrix}.    
\end{equation} 
Thus, differentiation with respect to $s$ readily implies that $\lambda = 0$ is a eigenvalue of the derivative $DF$ of $F$. In addition, the linearization of \eqref{eq:FOnlpert} around the time-independent solution $\phi(s)$ is
\begin{equation}
\label{eq:FOlpert}
\left\{
\begin{aligned}
\partial_t V &= \cA_c^s V, & z &\in \R, \, t > 0, \\ 
V(z,0)&=V_0(z),  & z &\in \R ,
\end{aligned}
\right.
\end{equation}
where
\[
\cA_c^s : = \begin{pmatrix} 0 & \Id \\-\cL_c^s & 2c\partial_z-\nu\Id\end{pmatrix},
\]
with
\[
\opl^s_c = \ocL_c^s  +c^2\partial^2_z-c\nu\partial_z+H \sin\psi(\cdot+s) \, \Id,\quad \mbox{and} \quad
\ocL_c^s=\left. \frac{d}{d\ep}\nabla\cE\left(\psi(\cdot+s)+\ep u\right)\right|_{\ep=0}.
\]
As before (see \cite{CMMP23}, \S 7), the operators $\cA^s_c:H^1\times L^2\to H^1\times L^2$ and $\cL^s_c:H^1\to L^2$ have dense domains, $D(\cA^s_c)=H^2\times H^1$ and $D(\cL^s_c)=H^2$, respectively. When $s=0$, we identify 
\[
\cA^0_c:= \cA_c, \qquad \mbox{and}\qquad \opl^0_c:=\opl_c.
\]
Indeed, it is not difficult to verify that if $T_l(s)$ denotes the left translation operator then it is an $H^1$-isometry and the generator of a $C_0$-semigroup (see \cite{CMMP23}, Section 7 and \cite{EN00}). As a consequence, all spectral results for the operator $\cA_c$ with $s = 0$ from last sections are applicable to the family of operators $\cA_c^s$ with $s \neq 0$. Actually, by Lemma 4.2.1 in \cite{KaPro13}, p. 87, the whole family of operators $\{s \in \R \,: \, \cA_c^s\}$ is \emph{isospectral}. In order to distinguish all the mathematical objects used until \S \ref{secsemiggen}, we add a superscript ``$s$'' for the corresponding mathematical object in the case where $s\neq 0$. For example, the projector operator $\cP_c^s$ and its range $\widetilde{X}^s$ which replace the projector operator $\cP_c$ and its range $\widetilde{X}$ used before for the case $s=0$. For details, see Capella \emph{et al.} \cite{CMMP23}.

Since the whole family of operators is isospectral, we conclude that the spectral gap is also preserved and the existence of a unique solution to the Cauchy linear problem \eqref{eq:FOlpert} in $H^1\times L^2$ is guaranteed by the action of the $C_0$-semigroup of quasicontractions generated by $\cA_c^s$ on the initial conditions. Moreover, when we restrict the admissible solutions of \eqref{eq:FOlpert} to the Hilbert subspace $\widetilde{X}^s$, the Theorem \ref{elcoro} implies the existence of constants $M\geq 1$ and $\widetilde{\omega}>0$ such that
\[
\|e^{t\tilde\cA_c^{s}}V_0\|_{H^1\times L^2}\leq Me^{-\widetilde{\omega} t}\|V_0\|_{H^1\times L^2}
\]
for every $V_0\in \widetilde{X}^s$. It is important to point out that the decaying rate bound $\widetilde{\omega}$ also is \emph{ independent of $s$} because it only depends on the spectral gap and the $H^1$ and $L^2$ norms of $\psi(\cdot + s)$, which remain constant under translations.

Just like in our previous stability analysis \cite{CMMP23}, the proof of Theorem \ref{maintheorem} follows from a direct application of the following abstract nonlinear stability result.

\begin{theorem}[Lattanzio \emph{et al.} \cite{LMPS16}]
\label{LPstability}
Let $X$ be a Hilbert space and $I\subset \R$ be an open neighborhood of $s=0$. Assume that $F:\cD\subset X \rightarrow X$ and $\phi:I\subset \R\rightarrow \cD$ satisfies $F(\phi) = 0$. Suppose that $\cP^s$ is the projector onto $\{\phi'(s)\}^\perp_X$ and that there exist positive constants $C_0,s_0,M,\omega$ and $\gamma$ such that
\begin{itemize}
\item[\rm{(A$_1$)}]\label{H1} for every solution $V = V(t,V_0,s)$ to \eqref{eq:FOlpert}, there holds
\[
\|\cP^s V(t,V_0,s)\|_X\leq C_0e^{-\omega t}\|\cP^s V_0\|_X,
\] 
\item[\rm{(A$_2$)}]\label{H2} $\phi$ is differentiable at $s = 0$ with 
\[
 \|\phi(s) -\phi(0)-\phi'(0)s\|_{X}\leq C_0|s|^{1+\gamma},
\]
for $|s|<s_0$; and,
\item[\rm{(A$_3$)}]\label{H3} $F$ is differentiable at $\phi(s)$ for every $s\in(-s_0,s_0)$ with 
\begin{equation}
\label{eq:H3}
 \|F(\phi(s)+W) -F(\phi(s))-DF(\phi(s))W\|_{X}\leq C_0\|W\|_{X}^{1+\gamma},
 \end{equation}
for  $\|W\|_{X}\leq M$.
\end{itemize} 
Then there exists $\ep>0$ such that for any $W_0\in B_\ep(\phi(0)) \subset X$ there exists $s \in I$  and a positive constant C for which the solution $W(t;W_0)$ to the nonlinear system \eqref{eq:FOnlpert} satisfies 
\[ 
\|W(t,W_0)-\phi(s)\|_{X}\leq C\ \|W_0-\phi(0)\|_{X}\ e^{-\omega t}.
\]
\end{theorem}

\begin{remark}
Theorem \ref{LPstability} is an extension to Hilbert spaces of an early result by Sattinger \cite{Sat76} which establishes nonlinear stability from spectral stability by controlling the growth of the nonlinear terms via the variation of constants formula, but only in the case of a simple translation eigenvalue $\lambda = 0$ for the linearization. In such a case the manifold generated by the traveling wave is one-dimensional, the projection $\cP$ onto the eigenspace has rank one and the necessary nonlinear modulations of the perturbations of the wave pertain to translations alone.
\end{remark}

\subsection{Proof of Theorem \ref{maintheorem}}
Let us choose $\varep > 0$ sufficiently small, more precisely, take $0 < \varep < \min \{ \widetilde{\delta}, \tilde{c}_\delta^{''}\}$ where $\widetilde{\delta} > 0$ is the constant from Theorem \ref{thm:CMO07} and $0 < \tilde{c}''_\delta \ll 1$ is sufficiently small such that $|H| \leq \tilde{c}''_\delta$ implies that $|c| = O(|H|) < c'_\delta$ and the conclusions of Theorems \ref{lem:spectral_gap} and \ref{theospectralstab} and Lemma \ref{elcoro} hold for any arbitrary $\delta \in (0, \zeta(\nu))$. Therefore, the proof of Theorem~\ref{maintheorem} follows from Theorem~\ref{LPstability} upon verification of assumptions 
\hyperref[H1]{(A$_1$)} thru \hyperref[H3]{(A$_3$)}. Let $X = H^1\times L^2$, $D := H^2\times H^1$, $F$ as in \eqref{eq:F} and $\phi(s)$ as in \eqref{F-invariance}. Hence $F(\phi(s)) = 0$ for every $s\in\R$ by \eqref{F-invariance}. Moreover, condition 
\hyperref[H1]{(A$_1$)} is satisfied for the projectors $\cP_c^s$ due to the isospectrality of the family of block operators $\{\cA_c^s:H^1\times L^2\to H^1\times L^2 \, : \, s \in \R\}$ and by virtue of Theorem \ref{elcoro}.

In order to verify 
\hyperref[H2]{(A$_2$)}, first notice that since $\psi\in H^2$ is a real-valued smooth function, then there holds 
\[
\|\phi(s)-\phi(0)-\phi'(0)s\|_{H^1\times L^2} = \|\psi(\cdot + s)-\psi-(\partial_z\psi) s\|_{H^1}. 
\]
By the Taylor series' remainder integral representation and by Jensen's inequality, we have
\[
\begin{aligned}
\nld{\psi(\cdot +s)-\psi-(\partial_{z}\psi) s}^2 
&\leq \int_{\R}s^4\left|\int_0^1 (1-t)\, \partial^2_{z}\psi(x+ts)\, dt\right|^2dx\\
&\leq s^4 \int_\R\int_0^1 (1-t)^2\,\left( \partial^2_{z}\psi(x+ts)\right)^2\, dt\,dx.
\end{aligned}
\]
Now, by changing the order of integration, and recalling that every translation is a $L^2$-isometry, we get
\[
    \nld{\psi(\cdot +s)-\psi-(\partial_{z}\psi)s}^2 
    = s^4\nld{ \partial^2_{z}\psi}^2 \int_0^1(1-t)^2\,dt.
\]
A similar argument for $\partial_{z}\psi$ yields
\[
    \nld{\partial_{z}\psi(\cdot +s)-\partial_{z}\psi-(\partial^2_z\psi) s}^2 
    = s^4\nld{ \partial^3_{z}\psi}^2 \int_0^1(1-t)^2\,dt.
\]
Therefore, 
\[
\|\phi(s)-\phi(0)-\phi'(0)s\|_{H^1\times L^2} \leq \frac{s^2}{\sqrt{3}} \nhu{\partial_{z}^2 \psi},
\]
and \hyperref[H2]{(A$_2$)} follows with $\gamma=1$. 

Now let $W=(w_1,w_2) \in H^2\times H^1$ and for any $s \in \R$ let us denote the $s$-left translation of the traveling N\'eel wall profile, namely, $\psi(\cdot+s)$, by $\psi_s$. Hence,
\[
\begin{aligned}
F(\phi(s)+W) &=
\begin{pmatrix}
\omega_2\\
2c\partial_{z}\omega_2 - c^2\partial_z^2(\psi_s+\omega_1)-\nu\omega_2+c\nu\partial_z(\psi_s+\omega_1) - \nabla \cE(\psi_s +\omega_1) +H\cos(\psi_s +\omega_1)
\end{pmatrix},\\
F(\phi(s)) &= 
\begin{pmatrix}
0\\
 - c^2\partial_z^2\psi_s+c\nu\partial_z \psi_s - \nabla \cE(\psi_s) +H\cos \psi_s
\end{pmatrix},\\
DF(\phi(s))W &= \cA_c^s W = \begin{pmatrix}
    w_2 \\ 2c\partial_{z} w_2-\nu w_2- \cL_c^s w_1
\end{pmatrix}.
\end{aligned}
\]
Upon substitution of the expressions for $F(\phi(s)+W)$, $F(\phi(s))$, $\cA_c^sW$, and the left-shifted version of equation \eqref{eq:Lc}, we get
\[
\begin{aligned}
\nwse{F(\phi(\delta)+W) -F(\phi(\delta))-DF(\phi(\delta))W} &\\
&\hspace{-7cm}=\nld{\nabla\cE(\psi_s)-\nabla \cE(\psi_s +w_1)+H\cos(\psi_s+\omega_1)  -H\cos\psi_s+\cL_c^s w_1-c^2\partial_z^2 w_1+c\nu\partial_z \omega_1},\\
&\hspace{-7cm}=\nld{\nabla\cE(\psi_s)-\nabla \cE(\psi_s +w_1)+\ocL_c^s w_1+H\cos(\psi_s+\omega_1)  -H\cos\psi_s+H\sin \psi_c  w_1},\\
&\hspace{-7cm}\leq \nld{\nabla\cE(\psi_s)-\nabla \cE(\psi_s +w_1)+\ocL_c^s w_1}+\nld{H\cos(\psi_s+\omega_1)  -H\cos\psi_s+H\sin \psi_c  w_1},
\end{aligned}
\]
where the last equality follows from the triangle inequality. The first term on the right hand side of last inequality satisfies the estimate
\[
\nld{\nabla\cE(\psi_s)-\nabla \cE(\psi_s +w_1)+\ocL_c^s w_1}\leq C \|W\|_{H^1\times L^2}^2,
\]
for some constant $C>0$. Since the traveling N\'eel wall profile phase $\psi$ is as smooth as the static N\'eel wall profile phase $\brt$, the proof of this statement is the same as the proof performed on bounding the term $\nld{\nabla\cE(\brt_\delta)-\nabla \cE(\brt_\delta +w_1)+\cL^\delta w_1}$ from Theorem 2.3 in \cite{CMMP23}, word-by-word.

Finally,  the inequality $|\cos (\psi_s+w_1)-\cos\psi_s+w_1\cos\psi_s| \leq \tfrac{1}{2}w_1^2$ and Sobolev's embedding theorem yield
\[
\begin{aligned}
\nld{H\cos(\psi_s+\omega_1)  -H\cos\psi_s+H\sin \psi_c  w_1}^2 
&\leq \tfrac{1}{4}|H|^2 \nld{w_1^2}^2 \\
&\leq \tfrac{1}{4}|H|^2 \|w_1\|^2_\infty\nld{w_1}^2\\
&\leq\tfrac{1}{4}|H|^2\nhu{w_1}^4\\
&\leq\tfrac{1}{4}|H|^2\|W\|^4_{H^1\times L^2}
\end{aligned}
\]
Therefore, 
\hyperref[H3]{(A$_3$)} holds for $\gamma=1$ and some $C_0 > 0$ that depends on the magnitude of the external magnetic field $H$. This completes the proof.
\qed

\section{Discussion}
\label{secdiscuss}

In this paper we have proved the nonlinear stability of moving N\'eel walls in ferromagnetic thin films. These structures emerge when a weak external magnetic field is applied. The moving N\'eel wall is described by a wave-type dynamics equation for the phase of the in-plane magnetization. The speed of propagation of the moving wall, $c$, is of order $O(|H|)$, where $|H| \ll 1$ measures the intensity of the external field. After linearizing the model equations around a moving wall's phase, the resulting spectral problem for the perturbation turns out to be an order $O(|c|)$-perturbation of the linearized operator around the static N\'eel wall, which is spectrally stable. The main result relies on the establishment of resolvent-type estimates for the latter operator which, in turn, imply the stability of the linearized operator around the moving N\'eel wall upon application of standard perturbation theory for linear operators. Whence, Lumer-Phillips theorem, the theory of $C_0$-semigroups and a nonlinear iteration argument yield the nonlinear stability result.

The smallness of the applied magnetic field is an important feature to infer both the existence and the stability of a moving domain wall. An important question from the point of view of applications is whether there exists a threshold value for the magnetic field's intensity $H$ for which we lose the stability and/or the existence of the moving wall. If the magnetic field is too strong then we expect the instability to be so violent that we practically never see the moving wall evolve in time. This behaviour is, perhaps, associated to a bifurcation phenomena in which the intensity of the magnetic field plays the role of the bifurcation parameter. This is an open problem that warrants further investigations.

\section*{Acknowledgements}

The work of A. Capella and R. G. Plaza was partially supported by CONAHCyT, M\'exico, grant CF-2023-G-122. The work of L. Morales was supported by CONAHCyT, M\'exico, through the Program ``Estancias Postdoctorales por M\'exico 2022''.

\appendix
\section{Useful estimates}
\label{secappendix}

In this Section we state and prove some useful lemmata which are needed in the proof of the resolvent estimates from Theorem \ref{lem:spectral_gap}.

\begin{lemma}\label{lem:aux0_mod} 
Let $\nu>0$ and $\delta\in (0,\nu/2)$ be fixed. Assume that $\Gamma \subset \C$ is the square of side length $2\delta$ and center at the origin as defined in \eqref{defofcontourGamma}. Also, define $S:[0,\pi/2]\times \C\to \R$ as 
\[
S(\phi,\lambda) := |\lambda^*\cos^2\phi+ (\lambda+\nu)\sin^2\phi|.
\]
Then, for every $(\phi,\lambda)\in[0,\pi/2] \times\{\Re \lambda>-\delta\}\setminus \conv(\Gamma)$  there holds that
\[ 
S(\phi,\lambda)\geq \frac{\delta}{\sqrt{\delta^2 + \tfrac{\nu^2}{4}}}\Big(\frac{\nu}{2}-\delta\Big)>
 0.
\]
Moreover, if $\Im \lambda \neq 0$, then 
\[
S(\phi,\lambda)\geq \frac{\sqrt{2}|\Im\lambda||\Re \lambda+\nu/2|}{\nu/2+|\Im \lambda|}>0. 
\]
\end{lemma}
\begin{proof}
First, we let $S_1^2(\phi,\lambda) := (\Re \lambda + \nu \sin^2 \phi)^2$ and $S_2^2(\phi,\lambda)=(\Im \lambda)^2\cos^22\phi$, this implies from the definition of $S$ that $S^2 = S_1^2+S_2^2$. By letting $\phi (u)= \arcsin(\tfrac{1}{\sqrt{2}}\sqrt{u+1})$ we get
\begin{equation}\label{eq:f}
f^2(u,\lambda):=S_1^2(\phi(u),\lambda)+S_2^2(\phi(u),\lambda) =\left(\Re \lambda + \tfrac{\nu}{2}\left(1+u\right)\right)^2
+(\Im \lambda)^2u^2.
\end{equation}

Notice that the mapping $\phi:[-1,1]\to [0,\pi/2]$ is onto, hence bounding $S^2$ on the set $[0,\pi/2] \times\{\Re \lambda>-\delta\}\setminus \conv (\Gamma)$ is equivalent to bounding $f^2$  on the set $\Omega:=[-1,1] \times\{\Re \lambda>-\delta\}\setminus \conv(\Gamma)$.  
Indeed, $f^2>0$, otherwise both terms in \eqref{eq:f} must vanish, yielding that either $\lambda=-\nu(1+u)/2$ or $\Re \lambda = -\nu/2$, but neither of both conditions are met on the set $\{\Re \lambda>-\delta\}\setminus \intconv (\Gamma)$. Also, since $f(u,\cdot)$ grows linearly with $\lambda$ for $|\lambda|$ big enough, then there exists a compact neighborhood of $(0,0)$ where $f^2$ meets its global minimum. It is not hard to see that $f^2$ has no critical points inside $\Omega$, then its minimum is attained at $\partial\Omega$ which is divided into the sets
\[
\begin{aligned}
\Omega_1 &= \{(u,\lambda)\in \Omega\ |\ u=-1\}, \\
\Omega_2 &= \{(u,\lambda)\in \Omega\ |\ u=1\},\\
\Omega_3 &= \{(u,\lambda)\in \Omega\ |\ |\Im \lambda|\leq \delta,\ \Re \lambda = \delta\},\\
\Omega_4 &= \{(u,\lambda)\in \Omega\ |\ |\Im \lambda|\geq \delta,\ \Re \lambda = -\delta\},\\
\Omega_5 &= \{(u,\lambda)\in \Omega\ |\ |\Im \lambda|=\delta, \ |\Re \lambda|\leq \delta\}.
\end{aligned}
\]
By a simple evaluation, we readily get that
\[
\min_{\Omega_1\cup \Omega_3} f^2(u,\lambda) = \delta^2,\quad \mbox{and } \quad  \min_{\Omega_2} f^2(u,\lambda) =(\nu-\delta)^2+ \delta^2.
\]
Regarding the sets $\Omega_4$ and $\Omega_5$ as two-dimensional sets, standard analysis tools yield that the minimum values of  $f^2$ on each set are attained on their boundary. Indeed,
\[
\min_{\Omega_4\cup\Omega_5} f^2(u,\lambda) = f^2\left(-\delta(1\pm i),\frac{-\tfrac{\nu}{2}(\tfrac{\nu}{2}-\delta)^2}{\delta^2 + \tfrac{\nu^2}{4}}\right) = \frac{\delta^2(\tfrac{\nu}{2}-\delta)^2}{\delta^2 + \tfrac{\nu^2}{4}}.
\]
This completes the first part of the statement. 
Finally, when $\Im \lambda\neq 0$ the convexity of $S^2_1(\phi(\cdot),\lambda)$ and $S^2_2(\phi(\cdot),\lambda)$ plus its non negativeness imply that $S^2(\phi(\cdot),\lambda)\geq \min\{S^2_1(\phi(\cdot),\lambda),S^2_2(\phi(\cdot),\lambda)\}$, which is also a convex and non negative function whose minimum is attained in the set $\{u\in[-1,1] \, : \,  S^2_1(\phi(\cdot),\lambda) = S^2_2(\phi(\cdot),\lambda)\}$. Due to these simple expressions, we easily get that 
\[
S^2(\phi(\cdot),\lambda)\geq \min\{S^2_1(\phi(\cdot),\lambda),S^2_2(\phi(\cdot),\lambda)\}\geq \frac{2(\Im\lambda)^2(\Re \lambda+\nu/2)^2}{(\nu/2+|\Im \lambda|)^2}>0.
\]
\end{proof}

\begin{lemma}
\label{lem:aux1_mod}
Let $\nu>0$  and $\beta\in (0,1)$ be fixed. Assume that $\lambda \in \C$ is such that $\Re \lambda>-\beta \nu/2$ and $(\Im \lambda)^2>\Lambda_0+\Lambda_0^{1/2}(2-\beta)\nu$; then 
\[
|1-\Lambda_0^{-1/2}|\nu+\lambda||>\left(1-\frac{\beta}{2}\right)\Lambda_0^{-1/2}\nu.
\]
\end{lemma}

\begin{proof}
We observe that $|1-\Lambda_0^{-1/2}|\nu+\lambda||\geq 1$ since $\Lambda_0^{-1/2}|\nu+\lambda|>\Lambda_0^{-1/2}|\Im \lambda|>1$. Therefore,
\[
\begin{aligned}
|1-\Lambda_0^{-1/2}|\nu+\lambda||&=\Lambda_0^{-1/2}\sqrt{(\re \lambda + \nu)^2 + (\Im \lambda)^2}-1\\
&\geq \Lambda_0^{-1/2} \sqrt{\left(1-\frac{\beta}{2}\right)^2\nu^2 + \Lambda_0+ \Lambda_0^{1/2}(2-\beta)\nu} \, -1 \\
&=\Lambda_0^{-1/2}\left( \Lambda_0^{1/2} + \Big(1-\frac{\beta}{2}\Big)\nu \right)-1\\
&=\left(1-\frac{\beta}{2}\right)\Lambda_0^{-1/2}\nu.
\end{aligned}
\]
\end{proof}

\begin{lemma}
\label{lem:aux2_mod}
Let $\lambda\in \C$ be as in Lemma \ref{lem:aux1_mod}.  Define $a(\lambda)=|\nu+\lambda|\Lambda_0^{-1/2}$ and  
\[
f_{a(\lambda)}(\varphi) =(1-a(\lambda))\cos \varphi-(1+a(\lambda))\sin \varphi. 
\]
Then $|f_{a(\lambda)}(\varphi)|\geq  \frac{1}{2}\left(1-\frac{\beta}{2}\right)\Lambda_0^{-1/2}\nu$ for every $|\varphi|<\ep$, where 
\[
\ep:=\frac{2\left(2-\beta\right)\Lambda_0^{-1/2}\nu}{4\pi+\left(2+\pi\right)\left(2-\beta\right)\Lambda_0^{-1/2}\nu}>0.
\] 
\end{lemma}

\begin{proof}
First, we notice that the function $f_{a(\lambda)}(\varphi)$ is continuous and decreasing on its argument since $a(\lambda)>1$. Also,   $f_{a(\lambda)}(\varphi)<0$ for $\varphi\in((1-a(\lambda))/(1+a(\lambda)),0)$. Therefore, $|f_{a(\lambda)}(\varphi)|\geq -f_{a(\lambda)}(-\frac{1}{2}\sin^{-1} \ep)\geq0$. Since $\cos \varphi \geq \frac{2}{\pi}\varphi+1$ and $\frac{\pi}{2}\varphi\geq \sin^{-1} \varphi $ hold for $\varphi\in[0,\pi/2]$ and $\varphi \geq \sin \varphi$, we get
\[
\begin{aligned}
|f_{a(\lambda)}(\varphi)|&\geq -f_{a(\lambda)}\left(-\frac{1}{2}\sin^{-1} \ep\right)\\
& = (a(\lambda)-1)\cos \left(\frac{1}{2}\sin^{-1} \ep\right)-(1+a(\lambda))\sin \left(\frac{1}{2}\sin^{-1} \ep\right)\\
 & \geq (a(\lambda)-1)\left(1-\frac{\sin^{-1} \ep}{\pi}\right)-(1+a(\lambda)) \left(\frac{1}{2}\sin^{-1} \ep\right)\\
 & \geq (a(\lambda)-1)\left(1-\frac{\ep}{2}\right)- \frac{\pi}{4} (1+a(\lambda))\ep.
\end{aligned}
\]
This last inequality and Lemma \ref{lem:aux1_mod} imply that 
\[
\begin{aligned}
|f_{a(\lambda)}(\varphi)|&\geq a(\lambda)\left(1-\left(\frac{1}{2}+\frac{\pi}{4}\right)\ep\right) -\left(1-\left(\frac{1}{2}-\frac{\pi}{4}\right)\ep\right)\\
&\geq \left(1+\left(1-\frac{\beta}{2}\right)\Lambda_0^{-1/2}\nu\right)\left(1-\left(\frac{1}{2}+\frac{\pi}{4}\right)\ep\right) -\left(1-\left(\frac{1}{2}-\frac{\pi}{4}\right)\ep\right)\\
&\geq \left(1-\frac{\beta}{2}\right)\Lambda_0^{-1/2}\nu -\left[\frac{\pi}{2}+\left(\frac{1}{2}+\frac{\pi}{4}\right)\left(1-\frac{\beta}{2}\right)\Lambda_0^{-1/2}\nu\right]\ep\\
&\geq \frac{1}{2}\left(1-\frac{\beta}{2}\right)\Lambda_0^{-1/2}\nu,
\end{aligned}
\]
where the last inequality holds for $\ep$ chosen as in the statement of the lemma. This finishes the proof.
\end{proof}

\begin{lemma}\label{lem:aux3_mod}
    Let $\nu>0$ be fixed and  $\zeta(\nu)>0$ be as in Lemma \ref{thm:steady_result}. Also, let $\delta \in(0,\zeta(\nu))$ be fixed and define   
%
    \begin{equation}\label{eq:setG2}
        G_2 = \{\lambda\in \C \, :\,\Re \lambda>-\delta, \ \ |\Im \lambda|>\delta \}\cup\{\delta(t \pm i) \, : \, t\in (-1,1)\}.
    \end{equation}

    The function $M:[0,\pi/2]\times G_2\to \R$, given by
    \begin{equation}\label{eq:M}
M(\phi, \lambda)= \min\left\{\frac{1}{|\cos \phi-|\lambda+\nu|\Lambda_0^{-1/2}\sin \phi|}, \frac{|\lambda| +|\lambda+\nu|}{|\lambda^*\cos^2\phi + (\lambda+\nu)\sin^2\phi|}
\right\},
\end{equation}
 is uniformly bounded on its domain.
\end{lemma}
\begin{proof}
    In order to prove this, 
    %
    let $\beta\in (0,1)$ be a fixed constant such that $2\delta<\beta\nu$ and define $\ep>0$ as
\[
\ep=\frac{2\left(2-\beta\right)\Lambda_0^{-1/2}\nu}{4\pi+\left(2+\pi\right)\left(2-\beta\right)\Lambda_0^{-1/2}\nu}.
\]
Then, $[0,\pi/2]\times G_2$ is the union of the following three sets:
\[
\begin{aligned}
H_1 &= ([0,\pi/2]\setminus(\pi/4-\ep,\pi/4+\ep))\times G_2, \\
H_2 &=(\pi/4-\tfrac{1}{2}\sin^{-1}\ep,\pi/4+\tfrac{1}{2}\sin^{-1}\ep)\times \{\lambda\in G_2\,|\, (\Im \lambda)^2\leq\Lambda_0+\Lambda_0^{1/2}(2-\beta)\nu \},\\
H_3 &=(\pi/4-\tfrac{1}{2}\sin^{-1}\ep,\pi/4+\tfrac{1}{2}\sin^{-1}\ep)\times \{\lambda\in G_2\,|\, (\Im \lambda)^2>\Lambda_0+\Lambda_0^{1/2}(2-\beta)\nu \}
\end{aligned}
\]
(see Figures \ref{figresolvent2} and \ref{figresolvent5}).

Before proving that $M(\phi, \lambda)$ is uniformly bounded on $(\phi,\lambda)\in [0,\pi/2]\times G_2$, let us examine Figures \ref{figresolvent2} and \ref{figresolvent5}. Figure \ref{figresolvent2} shows the projections of the sets $H_1$ and $H_2\cup H_3$ (regions in light blue and in light green, respectively) onto the plane $\Im \lambda = -d$ for any constant $d>\delta$. The purple and blue curves represent the regions where $(\Re \lambda + \nu \sin^2 \phi)^2=0$ and $(\Im \lambda)^2\cos^22\phi=0$ in the plane $\Im \lambda = -d$. Both curves intersect at the point $(\pi/4,-\nu/2)$ which lies outside of the region of interest because $\delta>\zeta(\nu)\geq -\nu/2$. This fact implies that 
\[
|\lambda^*\cos^2\phi + (\lambda+\nu)\sin^2\phi|>0,
\] 
in every plane of the form $\Im \lambda = -d$ for any constant $d>\delta$.

\begin{figure}[t]
\begin{center}
\includegraphics*[scale=0.7]{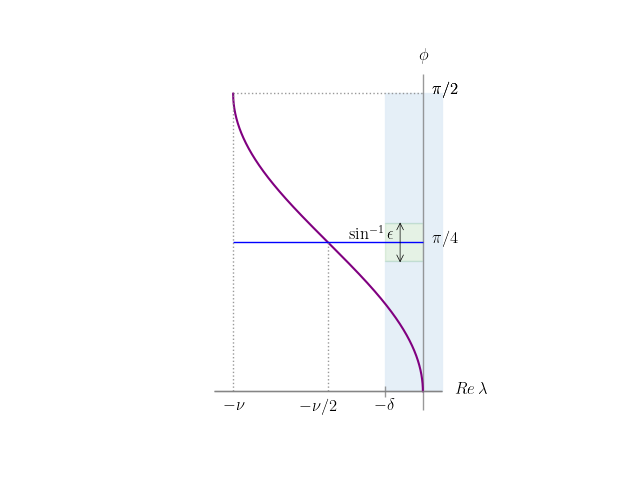} 
\end{center}
\caption{\label{figresolvent2} \small{Depiction of the sets $H_1$ and $H_2\cup H_3$, regions in light blue and light green colors, respectively, projected onto the plane $\Im \lambda = -d$ for any constant $d>\delta$} (color online).}
\end{figure}

\begin{figure}[t]
\begin{center}
\includegraphics*[scale=0.6]{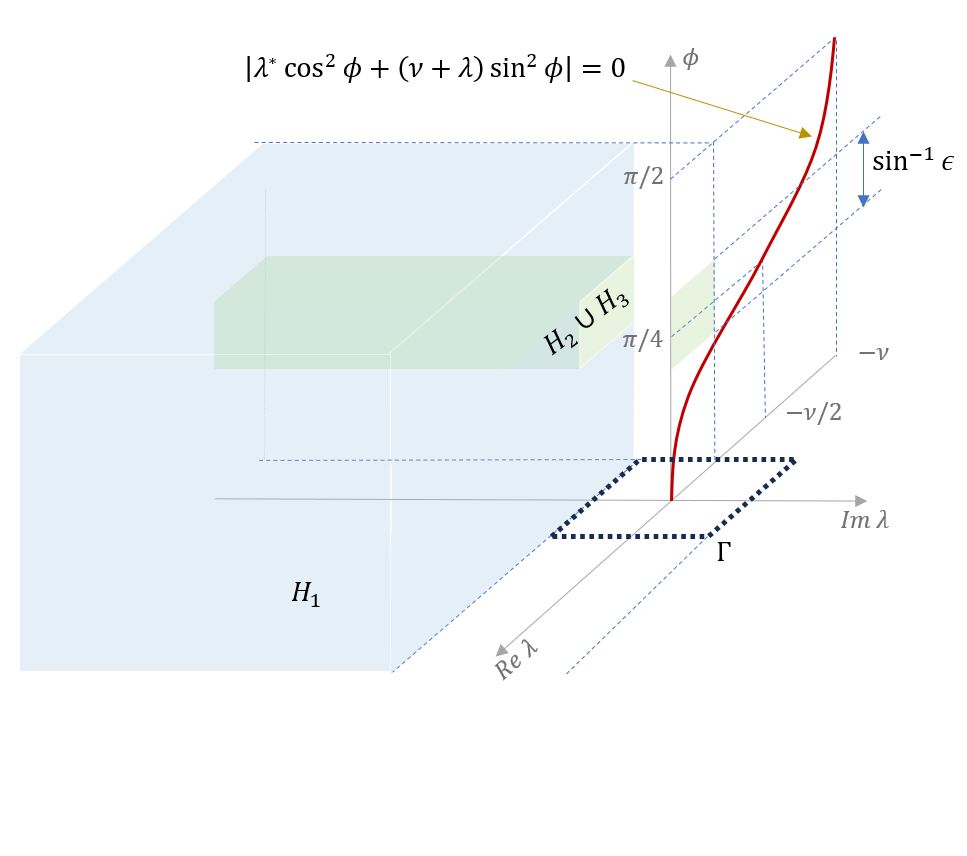} 
\end{center}
\caption{\label{figresolvent5} \small{Illustration of the sets $H_1$ and $H_2\cup H_3$ (in light blue and light green colors, respectively), in the region $\Im \lambda \leq 0$. Once again, the red curve and the black dotted curve are the sets where $|\lambda^*\cos^2\phi + (\lambda+\nu)\sin^2\phi|=0$ and the curve $\Gamma$, respectively (color online).}}
\end{figure}

Figure \ref{figresolvent5} shows the sets $H_1$ and $H_2\cup H_3$, regions in light blue and light green colors, respectively, in the region $\Im \lambda \leq 0$. Once again, the red curve and the black dotted curve represent the sets where $|\lambda^*\cos^2\phi + (\lambda+\nu)\sin^2\phi|=0$ and the curve $\Gamma$, respectively. This graphical representation makes evident that $M(\phi,\lambda)<\infty$ since this curve is outside $[0,\pi/2]\times G_2$. Despite of the boundedness of the term $|\lambda^*\cos^2\phi + (\lambda+\nu)\sin^2\phi|^{-1}$, the uniform boundedness of $M(\phi,\lambda)$ does not follow from it since 
\[
\left.\lim_{|\Im\lambda| \to \infty }\frac{|\lambda| +|\lambda+\nu|}{\sqrt{(\Re \lambda + \nu \sin^2 \phi)^2 + (\Im \lambda)^2\cos^22\phi}}\right|_{\pi/4}=\infty.
\]
This behavior justifies the necessity for the splitting of $[0,\pi/2]\times G_2$ into the sets $H_i$'s and the two resolvent estimates in Lemma \ref{lem:res_equations}. 

Notice that the set $H_2$ is empty provided that $\delta>\Lambda_0+\Lambda_0^{1/2}(2-\beta)\nu$. Also, due to the proximity of $\phi$ to $\pi/4$ in $H_2$ and $H_3$, we assume that, in those sets, $\phi = \pi/4 +\varphi$ with $|\varphi|<\tfrac{1}{2}\sin^{-1}\ep$. Thus, we easily get that 
\[
\ep^2 \geq \sin^2(2\varphi) = \cos^22(\pi/4+\varphi)=\cos^22\phi.
\] 

Now, we examine $M(\phi,\lambda)$ on each of the sets  $H_i$, $i =1,2,3$.
\begin{itemize}
\item[($H_1$)] 
Notice that $|\cos 2\phi| \geq\ep$ in $H_1$. Therefore, 
\[
\begin{aligned}
M(\phi,\lambda)&\leq \frac{(|\lambda| +|\lambda+\nu|)}{\sqrt{(\Re\lambda+\nu\sin^2\phi)^2 +(\cos(2\phi)\Im \lambda)^2} }\\
&\leq \frac{(|\lambda| +|\lambda+\nu|)}{\sqrt{(\Re\lambda+\nu\sin^2\phi)^2 +\ep^2(\Im\lambda)^2} }.
\end{aligned}
\]
Observe that the term on the right hand side is uniformly bounded since it is continuous as a function of $(\lambda,\phi)\in C_1\times ([0,\pi_4-\frac{1}{2}\sin^{-1}\ep]\cup[\pi_4+\frac{1}{2}\sin^{-1}\ep,\pi/2])$ and it is bounded as $|\lambda|\to \infty$.
\item[($H_2$)] In this set, Lemma \ref{lem:aux1_mod} and equation \eqref{eq:M} imply that
 \[
\begin{aligned}
M(\phi,\lambda) &\leq \left(\frac{\nu/2+|\Im \lambda|}{|\Im\lambda|}\right)\left(\frac{|\lambda| +|\lambda+\nu|}{|\Re \lambda+\nu/2|}\right)\\
&\leq \left(\frac{\nu/2+|\Im \lambda|}{|\Im\lambda|}\right)\left(\frac{|\Re\lambda| +|\Re\lambda+\nu| + 2|\Im \lambda|}{|\Re \lambda+\nu/2|}\right).\\
\end{aligned} 
\]
Thus, $M(\phi,\lambda)$ is uniformly bounded in $H_2$ since $\Im \lambda<(\Im \lambda)^2\leq\Lambda_0+\Lambda_0^{1/2}(2-\beta)\nu$.

\item[($H_3$)] By the selection of $\beta$, it follows that $\Re \lambda >-\beta\nu/2$ and $(\Im \lambda)^2>\Lambda_0+\Lambda_0^{1/2}(2-\beta)\nu$. Hence, the conclusion of Lemma \ref{lem:aux1_mod} holds. Moreover, the selection of $\ep$ implies that Lemma \ref{lem:aux2_mod} also holds. These observations immediately imply that
\[
\begin{aligned}
M(\phi,\lambda) &= M(\pi/4+\varphi,\lambda) \\
&\leq \frac{\sqrt{2}}{|(1-|\lambda+\nu|\Lambda_0^{-1/2})\cos \varphi-(1+|\lambda+\nu|\Lambda_0^{-1/2})\sin \varphi|}\\
&\leq  \frac{2\sqrt{2}\Lambda_0^{1/2}}{(1-\beta/2)\nu},
\end{aligned}
\] 
for $(\phi,\lambda)\in H_3$. Hence $M(\phi,\lambda)$ is uniformly bounded in $H_3$.
\end{itemize}

The combination of the estimates in these three cases proves the lemma.
\end{proof}




\end{document}